\documentclass[a4paper,oneside,11pt, reqno, psamsfonts]{amsart}
\usepackage{graphicx}
\usepackage{float}
\usepackage{pstricks}
\usepackage{amscd}
\usepackage{amsmath}
\usepackage{amsxtra}
\usepackage[T1]{fontenc}
\usepackage[utf8]{inputenc}
\usepackage{lipsum}
\usepackage{tikz}
\usetikzlibrary{arrows}
\usetikzlibrary{calc}

\usepackage{amsfonts,amssymb,amsmath,amsthm}
\usepackage{url}
\usepackage{enumerate}

\newcommand{\strutstretchdef}{\newcommand{\strutstretch}{\vphantom{\raisebox{1pt}{$\big($}\raisebox{-1pt}{$\big($}}}}
\theoremstyle{plain}
\newtheorem{theorem}{Theorem}[section]

\newtheorem{lemma}[theorem]{Lemma}
\newtheorem{proposition}[theorem]{Proposition}
\newtheorem{corollary}[theorem]{Corollary}

\theoremstyle{definition}

\newtheorem*{problem*}{Problem}
\newtheorem{example}[theorem]{Example}

\theoremstyle{remark}
\newtheorem{remark}[theorem]{Remark}

\numberwithin{equation}{section}

\newlength{\struh}
\newlength{\textminustop}
\strutstretchdef
\usepackage{mathrsfs}
\usepackage{epsfig}
\usepackage[active]{srcltx}

\newcommand{\ncom}{\newcommand}
\ncom{\bq}{\begin{equation}}
\ncom{\eq}{\end{equation}}
\ncom{\beqn}{\begin{eqnarray*}}
\ncom{\eeqn}{\end{eqnarray*}}
\ncom{\beq}{\begin{eqnarray}}
\ncom{\eeq}{\end{eqnarray}}
\ncom{\nno}{\nonumber}
\ncom{\rar}{\rightarrow}
\ncom{\Rar}{\Rightarrow}
\ncom{\noin}{\noindent}
\ncom{\bc}{\begin{centre}}
\ncom{\ec}{\end{centre}}
\ncom{\sz}{\scriptsize}
\ncom{\rf}{\ref}
\ncom{\sgm}{\sigma}
\ncom{\Sgm}{\Sigma}
\ncom{\dt}{\delta}
\ncom{\Dt}{Delta}
\ncom{\lmd}{\lambda}
\ncom{\Lmd}{\Lambda}
\ncom{\eps}{\epsilon}
\ncom{\pcc}{\stackrel{P}{>}}
\ncom{\dist}{{\rm\,dist}}
\ncom{\sspan}{{\rm\,span}}
\ncom{\im}{{\rm Im\,}}
\ncom{\sgn}{{\rm sgn\,}}
\ncom{\ba}{\begin{array}}
\ncom{\ea}{\end{array}}
\ncom{\eop}{\hfill{{\rule{1.5mm}{1.5mm}}}}
\ncom{\eoe}{\hfill{{\rule{1.5mm}{1.5mm}}}}
\ncom{\eof}{\hfill{{\rule{1.5mm}{1.5mm}}}}
\ncom{\hone}{\mbox{\hspace{1em}}}
\ncom{\htwo}{\mbox{\hspace{2em}}}
\ncom{\hthree}{\mbox{\hspace{3em}}}
\ncom{\hfour}{\mbox{\hspace{4em}}}
\ncom{\hsev}{\mbox{\hspace{7em}}}
\ncom{\vone}{\vskip 2ex}
\ncom{\vtwo}{\vskip 4ex}
\ncom{\vonee}{\vskip 1.5ex}
\ncom{\vthree}{\vskip 6ex}
\ncom{\vfour}{\vspace*{8ex}}
\ncom{\norm}{\|\;\;\|}
\ncom{\integ}[4]{\int_{#1}^{#2}\,{#3}\,d{#4}}
\ncom{\inp}[2]{\langle{#1},\,{#2} \rangle}
\ncom{\Inp}[2]{\Langle{#1},\,{#2} \Langle}
\ncom{\vspan}[1]{{{\rm\,span}\#1 \}}}
\ncom{\dm}[1]{\displaystyle {#1}}

\begin{document}
\title[The Cauchy dual subnormality problem]{The Cauchy dual subnormality problem via \\ de Branges-Rovnyak spaces}

\author[S. Chavan, S. Ghara and Md. R. Reza]{Sameer Chavan, Soumitra Ghara and Md. Ramiz Reza}

\address{Department of Mathematics and Statistics\\
Indian Institute of Technology Kanpur, India}
   \email{chavan@iitk.ac.in}
 \email{sghara@iitk.ac.in}
 \email{ramiz@iitk.ac.in}


\thanks{The second author is supported by the SERB through the NPDF post-doctoral fellowship (Ref.No. PDF/2019/002724).}



\begin{abstract} The Cauchy dual subnormality problem (for short, CDSP) asks whether the Cauchy dual of a $2$-isometry is subnormal. In this paper, we address this problem for cyclic $2$-isometries.
In view of some recent developments in operator theory on function spaces (see \cite{AM, LGR}), one may recast CDSP as the problem of subnormality of the Cauchy dual $\mathscr M'_z$ of the multiplication operator $\mathscr M_z$ acting on a de Branges-Rovnyak space $\mathcal H(B),$ where $B$ is a vector-valued rational function. The main result of this paper characterizes the subnormality of $\mathscr M'_z$ on $\mathcal H(B)$ provided $B$ is a vector-valued rational function with simple poles. As an application, we provide affirmative solution to CDSP for the Dirichlet-type spaces $\mathscr D(\mu)$ associated with measures $\mu$ supported on two antipodal points of the unit circle.  
\end{abstract}

\maketitle

\section{Cauchy dual subnormality problem for $2$-isometries} The Cauchy dual subnormality problem (for short, CDSP) for $2$-isometries can be seen as a manifestation of the rich interplay between positive definite and negative definite functions on abelian semigroups. Indeed, CDSP can be considered as the non-commutative variant of the fact from the harmonic analysis on semigroups that the reciprocal of a Bernstein function $f : [0, \infty) \rar (0, \infty)$ is completely monotone (see \cite[Theorem 3.6]{SSV}). This fact turns out to be somewhat equivalent to the solution of CDSP for completely hyperexpansive weighted shifts (see \cite[Proposition 6]{At} for a generalization). Another early result 
towards the solution of CDSP asserts that the Cauchy dual of any concave operator is a hyponormal contraction (see \cite[Equation (26)]{S1}). Later this fact was generalized in \cite[Theorem 3.1]{Ch} by deducing power hyponormality of the Cauchy dual of any concave operator. Around the same time CDSP was settled affirmatively 
for $\Delta_T$-regular $2$-isometries in \cite[Theorem 3.4]{BS} and for $2$-isometric operator-valued weighted shifts in \cite[Theorems 2.5 and 3.3]{ACJS} (see also \cite[Corollary 6.2]{CJJS} for the solution for yet another subclass of $2$-isometries). Further, it was shown in \cite[Examples 6.6 and 7.10]{ACJS} that there exist $2$-isometric weighted shifts on directed trees (that include adjacency operators) whose Cauchy dual is not necessarily subnormal. Recently, a class of cyclic $2$-isometric composition operators without subnormal Cauchy dual has been exhibited in \cite[Theorem 4.4]{ACJS-1}. 

For a complex Hilbert space $\mathcal H,$ let $\mathcal B(\mathcal H)$ denote the $C^*$-algebra of bounded linear operators on $\mathcal H.$ A bounded linear operator $T$ in $\mathcal B(\mathcal H)$ is {\it cyclic} if there exists a vector $f \in \mathcal H,$  called a {\it cyclic vector} of $T,$ such that $$\bigvee \{T^n f : n \geqslant 0\}=\mathcal H.$$ We say that $T$ is {\it analytic} if $\cap_{n \geqslant 0}T^n\mathcal H = \{0\}.$
Following \cite{S}, the {\it Cauchy dual} $T'$ of a left-invertible $T \in \mathcal B(\mathcal H)$ is defined by $T'=T(T^*T)^{-1}.$  
An operator $T \in \mathcal B(\mathcal H)$ is a {\it $2$-isometry} if 
\beqn
I - 2T^*T + T^{*2}T^2 =0.
\eeqn
An operator $T$ is said to be a {\it concave operator} if $I - 2T^*T + T^{*2}T^2 \leqslant 0.$
Clearly, any $2$-isometry is a concave operator.
By \cite[Lemma 1]{R1}, a concave operator $T$ is {\it norm increasing} (that is, $T^*T \geqslant I$), and hence the Cauchy dual of $T$ is well-defined (see \cite{Ag-St, R1} for the basic properties of $2$-isometries and concave operators). An operator $T \in \mathcal B(\mathcal H)$ is said to be {\it subnormal} if  there exist a Hilbert space $\mathcal K$ containing $\mathcal H$ and a normal operator $N \in \mathcal B(\mathcal K)$ such that $Nx=Tx$ for every $x \in \mathcal H$ (the reader is referred to \cite{Co1} for an excellent exposition on the theory of subnormal operators). 
With these notions, one can now state 
the Cauchy dual subnormality problem. This problem asks {\it when is the Cauchy dual $T'$ of a cyclic $2$-isometry $T$ in $\mathcal B(\mathcal H)$ subnormal?}

For a finite positive Borel measure $\mu$ on the unit circle $\mathbb T,$ the {\it Dirichlet-type space} $\mathcal D(\mu)$ is defined by
$$D(\mu):= \left\{f\in \mathcal O(\mathbb D) : \int_{\mathbb D}|f^{\prime}(z)|^2 P_{\mu}(z) \,dA(z) < \infty\right\},$$ where
$P_{\mu}(z)$ denotes the \emph{Poisson integral} $\int_{\mathbb T} \frac{1-|z|^2}{|z-\zeta|^2}d\mu(\zeta)$ of the measure $\mu,$ $dA(z)$ denotes the \emph{normalized Lebesgue area measure} on the open unit disc $\mathbb D$  and $\mathcal O(\mathbb D)$ denotes the space of complex valued holomorphic functions on  $\mathbb D.$
It is well-known that any cyclic analytic $2$-isometry is unitarily equivalent to the operator $\mathscr M_z$ of multiplication by the coordinate function $z$ on a Dirichlet-type space $D(\mu)$ for some finite positive Borel measure $\mu$ on the unit circle $\mathbb T$  (see \cite[Theorems 3.7 and 5.1]{R1}).
In view of this model theorem and the fact that any $2$-isometry decomposes into a direct sum of an analytic $2$-isometry and a unitary (see \cite[Proposition 3.4]{S}), the forgoing problem is equivalent to the following variant of CDSP.
\begin{problem*}[CDSP] Classify all finite positive Borel measures $\mu$ on the unit circle $\mathbb T$ for which the Cauchy dual $\mathscr M'_z$ of the multiplication operator $\mathscr M_z$ on the Dirichlet-type space $\mathscr D(\mu)$ is subnormal.
\end{problem*}

The counter-example to CDSP, as constructed in \cite[Section 4]{ACJS-1}, is a cyclic $2$-isometry $T$ for which the range of $T^*T-I$ is of infinite dimension. This construction poses the following question: Whether or not CDSP has affirmative answer for all cyclic $2$-isometries for which $T^*T-I$ is of finite rank? It turns out that all such $2$-isometries arise as the multiplication operator $\mathscr M_z$ acting on a Dirichlet-type space associated with a positive finitely supported measure (see Theorem \ref{lem-finite-rank-rep}). Since every analytic norm increasing operator is unitarily equivalent to the multiplication operator $\mathscr M_z$ acting on a de Branges-Rovnyak space (see \cite[Theorem 4.6]{LGR} and Lemma \ref{expansion-pd}),
the above problem for Dirichlet-type spaces associated with finitely supported measures can be addressed by examining de Branges-Rovnyak spaces of finite rank. This is precisely the purpose of the present paper.

\section{Main Results}

For complex separable Hilbert spaces $\mathcal D$ and $\mathcal E,$ let 
$\mathcal B(\mathcal D, \mathcal E)$ denote the Banach space of bounded linear transformations from $\mathcal D$ into $\mathcal E.$ 
The {\it Schur class} $\mathcal S(\mathcal D, \mathcal E)$ is given by 
\beqn
\mathcal S(\mathcal D, \mathcal E) = \{B : \mathbb D \rar \mathcal B(\mathcal D, \mathcal E) \,|\, B\mbox{~is~holomorphic~with~}\sup_{z \in \mathbb D} \|B(z)\|_{_{\mathcal B(\mathcal D, \mathcal E)}} \leqslant 1\}. 
\eeqn  
In case $\mathcal D = \mathbb C = \mathcal E,$ the Schur class is equal to the closed unit ball of 
 the space $H^{\infty}(\mathbb D)$ of bounded holomorphic functions on the unit disc $\mathbb D$. 
For any $B \in \mathcal S(\mathcal D, \mathcal E),$  the {\it de Branges-Rovnyak space} $\mathcal H(B)$ is the reproducing kernel Hilbert space associated with the $\mathcal B(\mathcal E)$-valued positive semi-definite kernel 
\beqn 
\kappa_B(z, w) = \frac{I_{\mathcal E}-B(z)B(w)^*}{1-z\overline{w}}, \quad z, w \in \mathbb D.
\eeqn
The kernel $\kappa_B$ is {\it normalized} if $\kappa_B(z, 0)=I_{\mathcal E}$ for every $z \in \mathbb D$. This is equivalent to $B(0)=0.$ 
If $\mathcal D = \mathbb C = \mathcal E,$ then we refer to $\mathcal H(B),$ to be denoted by $\mathcal H(b),$ as the {\it classical} de Branges-Rovnyak space (refer to \cite{Sa-0} and \cite{FM} for the basic theory of classical de Branges-Rovnyak spaces). 
Following \cite[Definition 1.1]{AM},  the spaces $\mathcal H(B)$ are referred to as {\it finite rank de Branges-Rovnyak spaces} provided $\mathcal D$ is finite dimensional and $\mathcal E$ is of dimension $1$. More generally, for a positive integer $k$ and $B \in 
\mathcal S({\mathcal D}, \mathcal E),$
the de Branges-Rovnyak space $\mathcal H(B)$ is said to be of of 
{\it rank} $k$ if 
\beqn
 k ~=& \inf \big\{  \dim \,\tilde{\mathcal D} : \mbox{there exist a Hilbert space}~\tilde{\mathcal D} ~\mbox{and~}\tilde{B} \in 
\mathcal S(\tilde{\mathcal D}, \mathcal E) \\
& \mbox{such that~} \mathcal H(B) = \mathcal H(\tilde{B})~\mbox{with equality of norms}\big\},
\eeqn
where $\dim$ denotes the Hilbert space dimension.
Assume that $\mathcal D =\mathbb C^k$ and $\mathcal E =\mathbb C.$ Then, for $B=(b_1, \ldots, b_k) \in \mathcal S(\mathcal D, \mathcal E)$, $z$ is a {\it multiplier}, that is, $zf \in \mathcal H(B)$ whenever $f \in \mathcal H(B)$ if and only if $$\int_{\mathbb T} \log\big(1-\textstyle \sum_{j=1}^{k}|\tilde{b}_j|^2\big)d\theta > -\infty,$$ 
where $d\theta$ denotes the normalized Lebesgue measure on the unit circle and $\tilde{b}_j(\zeta)$ denotes the non-tangential boundary value of $b_j$ at $\zeta \in \mathbb T$ 
(see \cite[Theorem 5.2]{AM} and \cite[Corollary 20.20]{FM}). In case of $k=1,$ by de Leeuw–Rudin theorem, $z$ is a multiplier if and only if $b$ is a non-extreme point of the unit ball of $H^{\infty}(\mathbb D)$ (see \cite[Theorem 12]{LR}).

Let $B \in \mathcal S(\mathcal D, \mathcal E).$
If $z$ is a multiplier, then by the closed graph theorem, the operator $\mathscr M_z$ of multiplication by $z$ on $\mathcal H(B)$ is bounded. In this case, $\mathscr M_z$ is always a norm increasing operator, that is, 
\beq \label{n-increasing}
\|\mathscr M_zf\| \geqslant \|f\|, \quad f \in \mathcal H(B)
\eeq
(see \cite[Theorem 2.1]{BB} and 
\cite[Pg 17]{LGR}). Conversely, every analytic norm increasing operator is unitarily equivalent to the operator $\mathscr M_z$ of multiplication by $z$ acting on a de Branges-Rovnyak space $\mathcal H(B)$ for some Schur-class function $B$  (see \cite[Theorem 4.6]{LGR} and Lemma \ref{expansion-pd}).

Here is the main result of this paper concerning the CDSP for a family of finite rank de Branges-Rovnyak spaces.
\begin{theorem} \label{higher-r-thm}
Let $B=(b_1, \ldots, b_k) \in \mathcal S(\mathbb C^k, \mathbb C)$ be such that $B(0)=0,$ where $$b_j(z) =\frac{p_j(z)}{\prod_{j=1}^k (z-\alpha_j)}$$ for polynomials $p_j$ of degree at most $k$ and distinct numbers $\alpha_1, \cdots, \alpha_k$ in $\mathbb C \setminus \overline{\mathbb D}.$
For $r=1, \ldots, k,$
let $a_r=  \prod_{1 \leq t \neq r \leq k}(\alpha_r - \alpha_t).$ 
 Assume that the operator $\mathscr M_z$ of multiplication by $z$ on the de Branges-Rovnyak space $\mathcal H(B)$ is bounded. 
Then the Cauchy dual $\mathscr M'_z$ of $\mathscr M_z$ is subnormal if and only if the matrix 
\beqn
\sum_{r, t=1}^k  \Big(\frac{1}{a_r \overline{a_t}}   \sum_{j=1}^k p_{j}(\alpha_r) \overline{p_{j}(\alpha_t)} \Big) (1-\frac{1}{\alpha_r \overline{\alpha}_t})^l \Big(\!\Big(\frac{1}{\alpha^{m+2}_r \overline{\alpha}^{n+2}_t}\Big)\!\Big)_{m, n \geqslant 0} ~\mbox{~}
\eeqn
is formally positive semi-definite for every  $l \geqslant 1.$
\end{theorem}
\begin{remark}
One of the essential ingredients in the proof of Theorem \ref{higher-r-thm} is Theorem \ref{main-thm-0}. The latter theorem can be used to extend Theorem \ref{higher-r-thm} to the case in which the rational symbol $B$ admits a pole of multiplicity bigger than $1.$  However, the characterization in this case is not as elegant as we get when $B$ admits simple poles.
\end{remark}

As the first application of Theorem \ref{higher-r-thm}, we obtain a handy criterion for the subnormality of the Cauchy dual of the multiplication operator under consideration.
\begin{corollary} \label{coro-p-zero}
Assume that  the hypotheses of Theorem \ref{higher-r-thm} hold. If
\beq \label{orthogonality-p}
\sum_{j=1}^k p_j(\alpha_r) \overline{p_j(\alpha_t)} = 0, \quad 1 \leqslant r \neq t \leqslant k,
\eeq
then $\mathscr M'_z$ is subnormal.
\end{corollary}

The following theorem provides an affirmative solution to CDSP for the Dirichlet-type spaces $\mathscr D(\mu)$ associated with measures $\mu$ supported on two antipodal points of the unit circle.
\begin{theorem} \label{rank-2-zeta}
For $\zeta \in \mathbb T$ and nonnegative numbers  $c_1,$ $c_2,$ let $\mathscr D(\mu)$ be the Dirichlet-type space  associated with the measure $\mu=c_1 \delta_{\zeta} + c_2 \delta_{-\zeta}$ and let $\mathscr M_z$ be the multiplication operator on $\mathscr D(\mu).$ Then the Cauchy dual $\mathscr M'_z$ of $\mathscr M_z$ is a subnormal contraction. 
\end{theorem}

Here is the layout of the paper. In Section 3, we consider the Cauchy dual subnormality problem in the set-up of functional Hilbert spaces  and apply its solution  to general de Branges-Rovnyak spaces (see Theorems \ref{EFCDSP} and \ref{main-thm-0}). 
In Section 4, we prove Theorem \ref{higher-r-thm} and Corollary \ref{coro-p-zero}, as stated above, and discuss some of their consequences (see  
Corollaries \ref{coro-posi} and \ref{coro-p-zero-iff}). In Section 5, we provide a complete solution to the CDSP for the classical de Branges-Rovnyak spaces (see Theorem \ref{main-thm}).
In Section 6, we characterize all cyclic analytic $2$-isometries $T$ for which the defect operator $\Delta_T:=T^*T-I$ is of finite rank. This result allows us to classify all $\Delta_T$-regular operators in the class of cyclic analytic $2$-isometries (see Theorem \ref{lem-finite-rank-rep}). Further, we identify the Dirichlet-type spaces associated with finitely supported measures with a de Branges-Rovnyak space $\mathcal H(B)$ and describe an algorithm to compute the symbol $B$ in this case (see Theorem \ref{Diri-de B}).
In Section 7, we  
apply the results obtained in the previous sections to  prove Theorem \ref{rank-2-zeta}. 
In this analysis, we arrive at a precise formula for the symbol $B$ of the rank $2$ de Branges-Rovnyak space $\mathcal H(B),$  which coincides with the Dirichlet-type space associated with a measure supported at anti-podal points (see Proposition \ref{Sarason-2}).

\section{The Cauchy dual subnormality problem in de Branges-Rovnyak spaces}


Let $\mathcal E$ be an auxiliary complex separable Hilbert space.
Let $\mathscr H_\kappa$ be a reproducing kernel Hilbert space of $\mathcal E$-valued holomorphic functions defined on the unit disc and let $\kappa : \mathbb D \times \mathbb D \rar B(\mathcal E)$ be the reproducing kernel for $\mathscr H_\kappa,$ that is, $\kappa(\cdot, w)x \in \mathscr H_\kappa$ and 
\beqn
\inp{f}{\kappa(\cdot, w)x}_{\mathscr H_\kappa} = \inp{f(w)}{x}_{\mathcal E}, \quad f \in \mathscr H_\kappa, ~w \in \mathbb D, ~x \in \mathcal E.
\eeqn
In the remaining part of this section, we assume the following:
\begin{enumerate}
\item[(A1)] $z$ is a multiplier for $\mathscr H_\kappa,$ that is, $zf \in \mathscr H_\kappa$ for every $f \in \mathscr H_\kappa,$
\item[(A2)] $\kappa$ is normalized at the origin, that is, $\kappa(z, 0)=I_\mathcal E$ for every $z \in \mathbb D,$
\item[(A3)] Under the assumptions (A1) and (A2), the orthogonal complement of $\{zf : f \in \mathscr H_\kappa\}$ is spanned by the space of $\mathcal E$-valued constant functions.
\item[(A4)] Under the assumption (A1), $\mathscr M_z$ is left-invertible.
\end{enumerate} 
We refer to the pair $(\mathscr H_\kappa, \mathcal E)$ satisfying (A1)-(A4) as the {\it functional Hilbert space}. The reader is referred to \cite{CS, BCR, AM, SSV, Sz} for the basics of vector-valued reproducing kernel Hilbert spaces, positive semi-definite kernels, Hausdorff moment problem, complex moment problem, completely monotone functions and related notions.
\begin{remark} \label{rmk-analytic}
By (A1) and the closed graph theorem, the operator $\mathscr M_z$ of multiplication by $z$ on $\mathscr H_\kappa$ is bounded. Further, since $\mathscr H_\kappa$ consists of $\mathcal E$-valued holomorphic functions, $\mathscr M_z$ is analytic.  
It is also clear from (A2) and (A3) that 
$\ker \mathscr M^*_z$ is spanned by the space of $\mathcal E$-valued constant functions and $\|x\|_{\mathscr H_\kappa} = \|x\|_{\mathcal E}$ for every $x \in \mathcal E.$ The {\it  backward shift operator}
 $\mathscr L_z$ given by 
\beqn
(\mathscr L_zf)(w) = \frac{f(w)-f(0)}{w}, \quad f \in \mathscr H_\kappa,~ w \in \mathbb D
\eeqn
is a well-defined bounded linear operator 
on $\mathscr H_\kappa$ that satisfies
\beq
\label{Lz-Cauchy}
\mathscr M'^*_z = \mathscr L_z.
\eeq
Indeed, in view of the closed graph theorem, it suffices to check that $\mathscr L_zf$ belongs to $\mathscr H_\kappa$ for every $f \in \mathscr H_\kappa.$ Since $\mathscr M_z$ is left-invertible, the range of $\mathscr M_z$ is closed. By (A2) and (A3), $$\{zf : f \in \mathscr H_\kappa\}=\mathscr H_\kappa \ominus \{\kappa(\cdot, 0)x : x \in \mathcal E\}=\{g \in \mathscr H_\kappa : g(0)=0\}.$$
This shows that $\mathscr L_z$ is bounded. 
Moreover, $\mathscr M'^*_z \mathscr M_z = I=\mathscr L_z \mathscr M_z.$ 
This together with $\mathscr M'^*_z x = 0=\mathscr L_z x,$ $x \in \mathcal E,$ yields \eqref{Lz-Cauchy}. 
\end{remark}

For a nonnegative integer $n,$ let $e_n : \mathbb D \rar B(\mathcal E)$ be given by
\beqn
e_n(z)x = \frac{\overline{\partial}^n \kappa(z, w)x}{n!}\Big|_{w=0}, \quad z \in \mathbb D, ~x \in \mathcal E,
\eeqn
where $\overline{\partial}^n\kappa(z, w)$ denotes the $n$th partial derivative of $\kappa(z, w)$ with respect to $\overline{w}.$ After interchanging the roles of analytic and coanalytic functions, it is easy to deduce from \cite[Lemma 4.1]{CS} that 
\beq \label{ip-with-en}
e_n(\cdot)x  \in \mathscr H_\kappa ~\mbox{and~} \inp{f}{e_n(\cdot)x} = \Big\langle \frac{({\partial}^nf)(0)}{n!},\, x  \Big \rangle, \quad f \in \mathscr H_\kappa, ~x \in \mathcal E.
\eeq
For integers $m, n \geqslant 0,$ we use the short notation ${\partial^{m}\overline{\partial}^{n} \kappa(0, 0)}$ to denote ${\partial^{m}\overline{\partial}^{n} \kappa(w, w)}|_{w=0}.$
Although the following is well-known (see \cite[Lemma 4.1]{CS} and \cite[Remark 2]{BKM}), 
we include a proof for the sake of completeness. 
\begin{proposition} \label{pd-k-m}
Let $\kappa : \mathbb D \times \mathbb D \rar B(\mathcal E)$ be a kernel function given by 
\beqn
\kappa(z, w) = \sum_{m, n =0}^{\infty} A_{m, n} z^m \overline{w}^n, \quad z, w \in \mathbb D,
\eeqn
where $A_{m, n} \in B(\mathcal E).$ Then $\kappa$ is a positive semi-definite kernel if and only if the matrix $A=\Big(\!\!\Big(A_{m, n}\Big)\!\!\Big)_{m, n \geqslant 0}$ is formally positive semi-definite, that is, $\Big(\!\!\Big(A_{m, n}\Big)\!\!\Big)_{0 \leqslant m, n \leqslant k}$ is  positive semi-definite for every nonnegative integer $k.$
\end{proposition}
\begin{proof} If $\kappa$ is positive semi-definite, then  
by \cite[Lemma 4.1(c)]{CS}, the matrix $B=\Big(\!\!\Big({\partial^{m}\overline{\partial}^{n} \kappa(0, 0)}\Big)\!\!\Big)_{m, n \geqslant 0}$ is formally positive semi-definite. If $C=\Big(\!\!\Big(\frac{\delta_{m, n}}{n!}\Big)\!\!\Big)_{m, n \geqslant 0},$ then $A = C^*BC,$ and hence 
$A$ is also formally positive semi-definite. This 
yields the necessity part. To see the sufficiency part, note that 
\beqn
\kappa_l(z, w) := \sum_{m, n=0}^l A_{m, n} z^m \overline{w}^n=X_l(z)AX_l(w)^*, \quad z, w \in \mathbb D, ~l \geqslant 0,
\eeqn
where $X_l(z)=(I_{\mathcal E}, zI_{\mathcal E}, \ldots, z^lI_{\mathcal E}, 0, \ldots) : \ell^2(\mathcal E) \rar \mathcal E.$ Thus if $A$ is formally positive semi-definite, then $\kappa_l$ is a positive semi-definite kernel for every integer $l \geqslant  0.$ Since the pointwise limit of positive semi-definite kernels is again positive semi-definite, the proof is complete.
\end{proof}

We need the following characterization of the subnormality of the Cauchy dual of multiplication operator $\mathcal M_z$ acting on functional Hilbert spaces in the proof of Theorem \ref{main-thm-0}.
\begin{theorem} \label{EFCDSP} Let $(\mathscr H_\kappa, \mathcal E)$ be a functional Hilbert space and let $\mathscr M_z$ be the operator of multiplication by $z$ on $\mathscr H_\kappa.$ Assume that $\mathscr M_z$ is norm-increasing. 
Then the following statements are equivalent:
\begin{enumerate}
\item[(i)] The Cauchy dual $\mathscr M'_z$ of $\mathscr M_z$ is subnormal.
\vskip.1cm
\item[(ii)] $\displaystyle \sum_{j=0}^k (-1)^j {k \choose j} \Big(\!\!\Big(\frac{\partial^{m + j}\overline{\partial}^{n+j} \kappa(0, 0)}{(m+j)!(n+j)!}\Big)\!\!\Big)_{m, n \geqslant 0}$ is formally positive semi-definite for every $k \geqslant 0$.
\end{enumerate}
If, in addition, $\mathcal E$ is one dimensional, then (1) and (2) are  equivalent to 
\begin{enumerate}
\item[(iii)]  $\big\{\frac{\partial^m\overline{\partial}^n \kappa(0, 0)}{m!n!}\big\}_{m, n \geqslant 0}$ is a complex moment sequence.
\end{enumerate}
\end{theorem}

\begin{proof}
Assume that $\mathscr M_z$ is norm-increasing.
To see the equivalence of (i) and (ii),  note that $\mathscr M'_z$ is a contraction, and hence by Agler's criterion \cite[Theorem 3.1]{Ag}, $\mathscr M'_z$ is subnormal if and only if for every integer $k \geqslant 0,$$$\sum_{j=0}^k (-1)^j {k \choose j} \mathscr M'^{*j}_z \mathscr M'^{j}_z \geqslant 0.$$ One may now infer from \cite[Theorem 2.15]{S} that $\mathscr M'_z$ is subnormal if and only if for every integer $k \geqslant 0,$
\beqn
\sum_{j=0}^k (-1)^j {k \choose j} \sigma^{*j}\kappa(z, w) 
\eeqn
is formally positive semi-definite,  where $\sigma^{*j}$ is given by
\beqn
\sigma^{*j} \kappa(z, w)  = \sum_{m, n \geqslant 0} \frac{\partial^{m + j}\overline{\partial}^{n+j} \kappa(0, 0)}{(m+j)!(n+j)!} z^m \overline{w}^n, \quad j \geqslant 0.
\eeqn 
The equivalence of (i) and (ii) is now immediate from Proposition \ref{pd-k-m}.

To see the remaining implications, let $\mathscr L_z$ be the backward shift operator (see Remark \ref{rmk-analytic}). We claim that 
\beq \label{back-s} \mathscr M'_z (e_n(\cdot)x) = e_{n+1}(\cdot)x~\mbox{ for every}~ n \geqslant 0 ~\mbox{and}~ x \in \mathcal E. \eeq 
Note that for any $f \in \mathscr H_\kappa,$  by \eqref{ip-with-en} (applied twice), 
\beqn
\inp{f}{\mathscr L^*_z e_n(\cdot)x}_{\mathscr H_\kappa} &=& \big\langle \mathscr L_z(f),\,{e_n(\cdot)x}\big\rangle_{\mathscr H_\kappa} = \Big\langle \frac{({\partial}^n\mathscr L_z(f))(0)}{n!},\, x\Big\rangle_{\mathcal E} \\ &=& \Big \langle \frac{({\partial}^{n+1} f))(0)}{(n+1)!},\, x\Big \rangle_{\mathcal E} = \inp{f}{e_{n+1}(\cdot)x}_{\mathscr H_\kappa}.
\eeqn
Since $f$ is arbitrary, an application of \eqref{Lz-Cauchy} yields \eqref{back-s}. 

To see the implication (i)$\Rightarrow$(iii), note that by \eqref{back-s} and \eqref{ip-with-en}, for all integers $m, n \geqslant 0$ and $x, y \in \mathcal E,$
\beqn 
\inp{\mathscr M'^n_z x}{\mathscr M'^m_z y}_{\mathscr H_\kappa} &=&  \inp{e_n(\cdot)x}{e_m(\cdot)y}_{\mathscr H_\kappa}  \notag \\
&=&  \Big\langle \frac{({\partial}^me_n(\cdot)x)(0)}{m!},\, y\Big \rangle_{\mathcal E} \notag \\ &=&  \Big\langle \frac{\partial^m {\overline \partial}^n  \kappa(0, 0)}{m!n!}x,\,y\Big\rangle_{\mathcal E}.
\eeqn
This shows that 
\beq \label{coefficients-kernel}
P_{\mathcal E} \mathscr M'^{*m}_z \mathscr M'^n_z|_{\mathcal E} =  \frac{\partial^m {\overline \partial}^n  \kappa(0, 0)}{m!n!}, \quad m, n \geqslant 0,
\eeq
where $P_{\mathcal E}$ denotes the orthogonal projection of $\mathscr H_\kappa$ onto the space $\mathcal E$ of constant functions in $\mathscr H_\kappa.$
Assume now that $\mathscr M'_z$ is subnormal. By Bram's characterization of subnormality \cite[1.9 Theorem(a)(i)]{Co1}, 
there exists a semispectral measure $Q$ compactly supported in $\mathbb C$ such that
\beqn
\mathscr M'^{*m}_z \mathscr M'^n_z = \int z^n \overline{z}^m dQ(z), \quad m, n \geqslant 0.
\eeqn 
It now follows from \eqref{coefficients-kernel} that $\big\{\frac{\partial^m\overline{\partial}^n \kappa(0, 0)}{m!n!}\big\}_{m, n \geqslant 0}$ is a  complex moment sequence with the representing measure $P_{\mathcal E}Q(\cdot)|_{\mathcal E}.$ This gives the implication (1)$\Rightarrow$(3) (here we do not need the assumption that $\mathcal E$ is one dimensional).

To see the implication (iii)$\Rightarrow$(i),  assume that the dimension of $\mathcal E = \ker \mathscr M^*_z$ is $1.$ By \cite[Corollary 2.8]{S}, the Cauchy dual of $\mathscr M_z$ is cyclic with cyclic vector in $\mathcal E,$ and hence 
one may apply \cite[Theorem 35]{SS} to complete the verification of (iii)$\Rightarrow$(i). 
\end{proof}

We already noted that every analytic norm increasing operator can be modeled as the operator of multiplication by $z$ on a de Branges-Rovnyak space (see \cite[Proposition 2.1]{AM} and \cite[Theorem 4.6]{LGR} for generalizations). In view of the Shimorin's model theorem (see \cite{S}), the forgoing fact can be obtained using the reproducing kernel space techniques (see \cite{AM}).
\begin{lemma}
\label{expansion-pd} 
Let $(\mathscr H_\kappa, \mathcal E)$ be a functional Hilbert space and let $\mathscr M_z$ be the operator of multiplication by $z$ on 
$\mathscr H_{\kappa}.$ If $\mathscr M_z$ is norm increasing, then there exist a complex Hilbert space $\mathcal D$ and $B \in \mathcal S(\mathcal D, \mathcal E)$ such that 
$\mathscr H_\kappa$ coincides with $\mathcal H(B)$ with equality of norms.
\end{lemma}
\begin{proof}
Assume that $\mathscr M^*_z \mathscr M_z \geqslant I$ and let
$P_{\mbox{ran}\, \mathscr M_z}$ denote the orthogonal projection of $\mathscr H_\kappa$ onto the $\mbox{ran}\, \mathscr M_z.$ 
Note that
\beqn
\mathscr M_z \mathscr M^*_z - P_{\mbox{ran}\,\mathscr M_z} &=& \mathscr M_z \mathscr M^*_z - \mathscr M_z\mathscr M'^*_z \\ &=& \mathscr M_z (I  - (\mathscr M^*_z \mathscr M_z)^{-1})\mathscr M^*_z,
\eeqn
which is positive. 
It follows that for any $f \in \mathscr H_\kappa,$
\beqn
\|\mathscr M^*_zf\|^2   \geqslant   \|P_{
\mbox{ran}\,\mathscr M_z}f\|^2 = \|f-f(0)\|^2 = \|f\|^2 - \|f(0)\|^2_{\mathcal E},
\eeqn 
where we used the fact that $f-f(0)$ and $f(0)$ are orthogonal in $\mathscr H_\kappa$. 
Letting  $f = \sum_{j=1}^n \kappa(\cdot, \omega_j)c_j$ for $c_1, \ldots, c_n \in \mathcal 
E$ and $\omega_1, \ldots, \omega_n \in \mathbb D,$  note further that
\beqn
\sum_{i, j =1}^n  \overline{\omega}_i \omega_j  \inp{\kappa(\omega_j, \omega_i)c_i}{{c}_j}_{\mathcal E} & \geqslant & \sum_{i, j =1}^n  \inp{\kappa(\omega_j, \omega_i)c_i}{{c}_j}_{\mathcal E} - \sum_{i, j =1}^n  \inp{c_i}{c_j}_{\mathcal E} \\
&=& \sum_{i, j =1}^n  \inp{\big(\kappa(\omega_j, \omega_i)-I_{\mathcal E}\big)c_i}{{c}_j}_{\mathcal E}. 
\eeqn 
Thus $\eta(z, w):=z\overline{w} \kappa(z, w) - (\kappa(z, w) - I_{\mathcal E}),$ $z, w \in \mathbb D,$ is a positive semi-definite kernel. The existence of $B : \mathbb D \rar  B(\mathcal D, \mathcal E)$ such that $\eta(z, w)=B(z)B(w)^*,$ $z, w \in \mathbb D,$ now follows from the factorization theorem for positive semi-definite kernels (see \cite[Theorem 2.62]{AMc}).
Moreover, since $\eta(z, z) \leqslant I_{\mathcal E}$ 
 for every 
$z \in \mathbb D,$ $B \in \mathcal S(\mathcal D, \mathcal E).$ Finally, note that
$\kappa(z, w) = \frac{I_{\mathcal E}-B(z)B(w)^*}{1-z\overline{w}}$ for every $z, w \in \mathbb D.$ 
\end{proof}

In view of Lemma \ref{expansion-pd}, the Cauchy dual subnormality problem for analytic norm increasing operators reduces to the same problem for the operator of multiplication by $z$ on de Branges-Rovnyak spaces. 
The main result of this section characterizes subnormality of the Cauchy dual of the multiplication operator $\mathscr M_z$ on de Branges-Rovnyak spaces. 
\begin{theorem} \label{main-thm-0}
Let $B(z) = \sum_{j=1}^{\infty} B_j z^j $ belong to the Schur class $\mathcal S(\mathcal D, \mathcal E).$ Assume that the operator $\mathscr M_z$ of multiplication by $z$ on $\mathcal H(B)$ is bounded and the orthogonal complement of $\{zf : f \in \mathcal H(B)\}$ is spanned by the space of $\mathcal E$-valued constant functions.
Then the following statements are equivalent: 
\begin{enumerate}
\item[(i)] The Cauchy dual $\mathscr M'_z$ of $\mathscr M_z$ is subnormal.
\item[(ii)] The matrix $\sum_{j=0}^k (-1)^j {k \choose j} \big(\!\!\big(B_{m+1+j} \,B^*_{n+1+j}\big)\!\!\big)_{m, n \geqslant 0}$ is formally positive semi-definite for every $k \geqslant 1.$
\item[(iii)] There exists a $\mathcal B(\ell^2(\mathcal E))$-valued semi-spectral measure $F$ supported in $[0, 1]$ such that 
\beqn
\big(\!\!\big(B_{m+1+j} \,B^*_{n+1+j}\big)\!\!\big)_{m, n \geqslant 0} = \int_{0}^1 t^j F(dt), \quad j \geqslant 0.
\eeqn
\end{enumerate}
\end{theorem}
\begin{proof}
Let $\kappa_B$ denote the reproducing kernel for $\mathcal H(B).$ 
Note that 
\beqn 
\kappa_B(z, w) &=& \sum_{n=0}^{\infty} z^n\overline{w}^n I_{\mathcal E} - B(z) B(w)^* \sum_{n=0}^{\infty} z^n\overline{w}^n \\
&=& \sum_{n=0}^{\infty} z^n\overline{w}^n I_{\mathcal E} - \Big(\sum_{j, k=1}^{\infty} B_j B^*_k \, z^j \overline{w}^k \Big)  \sum_{n=0}^{\infty} z^n\overline{w}^n. 
\eeqn
Comparing the coefficients of $z^m\overline{w}^n,$ $m, n \geqslant 0,$ on the both sides, we obtain
\beq \label{formula-coeff}
\frac{\partial^m\overline{\partial}^n \kappa_B(0, 0)}{m!n!} &=& \delta_{m, n} I_{\mathcal E} - \sum_{j=1}^m \sum_{k=1}^n B_{ j}{B^*_{ k}} 
\delta_{m-j, n-k} \\
&=& \begin{cases} \delta_{m, n} I_{\mathcal E}  - \sum_{k=1}^n B_{ m-n+k} {B^*_{ k}} & \mbox{if~} m \geqslant n, \\
\delta_{m, n} I_{\mathcal E}  - \sum_{j=1}^m B_{ j} {B^*_{ n-m+j}} & \mbox{if~} m < n,
 \end{cases} \notag 
\eeq
where $\delta_{m, n}$ denotes the Kronecker delta and the convention that sum over the empty set is $0$ is used. Set
\beqn
 f_k(m, n) &:=& \sum_{j=0}^k (-1)^j {k \choose j} \frac{\partial^{m + j}\overline{\partial}^{n+j} \kappa_B(0, 0)}{(m+j)!(n+j)!}, \quad k \geqslant 1, ~m, n \geqslant 0. 
 \eeqn 
A routine verification using ${n \choose k} + {n \choose k-1} = {n+1 \choose k},$ $n \geqslant k$, shows that
 \beq \label{step-indu}
 f_{k+1}(m, n) = f_k(m, n)-f_k(m+1, n+1), \quad k \geqslant 1, ~m, n \geqslant 0. 
 \eeq
We verify by induction on $k \geqslant 1$  that  
\beq \label{indu-subn}
  f_k(m, n) =  \sum_{j=0}^{k-1} (-1)^j {k-1 \choose j} B_{m+1+j}B^*_{n+1+j}, \quad m, n \geqslant 0. \quad 
\eeq
By \eqref{formula-coeff}, for integers $m, n \geqslant 0,$
\beqn
f_1(m, n) = \frac{\partial^{m}\overline{\partial}^{n} \kappa_B(0, 0)}{m!n!} - \frac{\partial^{m + 1}\overline{\partial}^{n+1} \kappa_B(0, 0)}{(m+1)!(n+1)!} = 
B_{m+1}B^*_{n+1}. 
\eeqn
Thus \eqref{indu-subn} holds for $k=1.$  If \eqref{indu-subn} holds for $k\geqslant 1,$ then by \eqref{step-indu},
\beqn
&& f_{k+1}(m, n) \\ 
&=&  \sum_{j=0}^{k-1} (-1)^j {k-1 \choose j} B_{m+1+j}B^*_{n+1+j} -  \sum_{j=0}^{k-1} (-1)^j {k-1 \choose j} B_{m+2+j}B^*_{n+2+j} \\
&=& \sum_{j=0}^{k}  (-1)^j {k \choose j} B_{m+1+j}B^*_{n+1+j}.
\eeqn
This completes the verification of \eqref{indu-subn}.

To complete the proof, recall that 
$\mathscr M_z$ is norm increasing (see \eqref{n-increasing}) or equivalently $\mathscr M'_z$ is a contraction. One may now apply Theorem \ref{EFCDSP} together with \eqref{indu-subn} to obtain the equivalence of (i) and (ii). In view of the polarization technique (see \cite[Proof of Theorem 4.2]{J}), the equivalence of (ii) and (iii) follows from the solution of the Hausdorff moment problem (cf \cite[Proposition 6.11, Chapter 4]{BCR}) provided we show that $\big(\!\!\big(B_{m+1+j} B^*_{n+1+j}\big)\!\!\big)_{m, n \geqslant 0}$ defines a bounded linear operator on $\ell^2(\mathcal E)$ for every integer $j \geqslant 0.$ Since $B$ belongs to $\mathcal S(\mathcal D, \mathcal E),$
 the function $B(\cdot)x$ belongs to $H^2_{\mathcal E}(\mathbb D).$ Thus the linear map $C : \mathcal D \rar \ell^2(\mathcal E)$ given by $C(x) = (B_{k}x)_{k \geqslant 1}$  is well-defined.
 The boundedness of $C$ now follows either from the assumption that  $B \in \mathcal S(\mathcal D, \mathcal E)$ or from the closed graph theorem. Since $\big(\!\!\big(B_{m} B^*_{n}\big)\!\!\big)_{m, n \geqslant 1}=CC^*,$ the desired equivalence is immediate.
\end{proof}

As an application, we characterize the subnormality of the Cauchy dual of the multiplication operator $\mathscr M_z$ on finite rank de Branges-Rovnyak spaces. 
\begin{corollary}
For a positive integer $k,$ let $B(z) = \sum_{j=1}^{\infty} B_j z^j $ belong to the Schur class $\mathcal S(\mathbb C^k, \mathbb C).$ Assume that the operator $\mathscr M_z$ of multiplication by $z$ on $\mathcal H(B)$ is bounded.
Then the following statements are equivalent: 
\begin{enumerate}
\item[(i)] The Cauchy dual $\mathscr M'_z$ of $\mathscr M_z$ is subnormal.
\item[(ii)] The matrix $\sum_{j=0}^k (-1)^j {k \choose j} \big(\!\!\big(B_{m+1+j} \,B^*_{n+1+j}\big)\!\!\big)_{m, n \geqslant 0}$ is formally positive semi-definite for every $k \geqslant 1.$
\item[(iii)] There exists a $\mathcal B(\ell^2(\mathbb N))$-valued semi-spectral measure $F$ supported in $[0, 1]$ such that 
\beqn
\big(\!\!\big(B_{m+1+j} \,B^*_{n+1+j}\big)\!\!\big)_{m, n \geqslant 0} = \int_{0}^1 t^j F(dt), \quad j \geqslant 0.
\eeqn
\end{enumerate}
\end{corollary}
\begin{proof}
In view of Theorem \ref{main-thm-0}, it suffices to check that the orthogonal complement of $\{zf : f \in \mathcal H(B)\}$ is spanned by the space of constant functions. This is immediate from \cite[Theorem 1.3]{AM}.
\end{proof}

\section{The proof of the main theorem and its consequences}

We now complete the proof of the main result stated in Section 2.
\begin{proof}[Proof of Theorem \ref{higher-r-thm}]
Write $B(z) = \sum_{l=1}^{\infty} B_l z^l,$ $|z| < 1.$ Since $\alpha_1, \ldots, \alpha_k$ are distinct, by \cite[Proposition 2.1]{RZ}, for any polynomial $p$ of degree less than $k,$
\beqn
\frac{p(z)}{\prod_{i=1}^k(z-\alpha_i)} = \sum_{i=1}^k \frac{p(\alpha_i)}{a_i} \frac{1}{z-\alpha_i}.
\eeqn
Applying this to $\frac{p_j(z)}{z}$ for every $j=1, \ldots, k,$ we obtain
\beq \label{Ramiz}
b_j(z) = \frac{p_j(z)}{\prod_{i=1}^k(z-\alpha_i)} = \sum_{i=1}^k \frac{p_j(\alpha_i)}{\alpha_i a_i} \frac{z}{z-\alpha_i} = -\sum_{i=1}^k \frac{p_j(\alpha_i)}{\alpha_i a_i} \sum_{l=0}^{\infty} \frac{z^{l+1}}{\alpha^{l+1}_i}.
\eeq
For $j=1, \ldots, k,$ let $b_{j}(z)=\sum_{m=1}^{\infty}b_{j, m}z^m,$ $z \in \mathbb D$ and note that $B_{m}=(b_{1, m}, \ldots, b_{k, m}),$ $m \geqslant 1.$ 
By \eqref{Ramiz}, for any $s \geqslant 0$ and $m \geqslant 0,$  
\beqn
b_{j, m+1+s} = -\sum_{i=1}^k \frac{p_j(\alpha_i)}{a_i}  \frac{1}{\alpha^{m+s+2}_i},
\eeqn
and hence
\beq \label{prod-exp} 
\notag
B_{m+1+s} \,B^*_{n+1+s} &=& \sum_{j=1}^k b_{j, m+1+s} \,\overline{b_{j, n+1+s}} \\ \notag 
&=& \sum_{j=1}^k \Big(\sum_{r=1}^k \frac{p_j(\alpha_r)}{a_r}  \frac{1}{\alpha^{m+s+2}_r}\Big)  \Big(\sum_{t=1}^k \frac{\overline{p_j(\alpha_t)}}{\overline{a_t}}  \frac{1}{\overline{\alpha}^{n+s+2}_t}\Big) \\
&=&  \sum_{r, t=1}^k \Big(\frac{1}{a_r \overline{a_t}}   \sum_{j=1}^k p_{j}(\alpha_r) \overline{p_{j}(\alpha_t)} \Big) \frac{1}{\alpha^{m+2+s}_r \overline{\alpha}^{n+2+s}_t}.
\eeq
Note that 
\beqn
\sum_{s=0}^l (-1)^s {l \choose s} \frac{1}{(\alpha_r \overline{\alpha}_t)^s} = (1-\frac{1}{\alpha_r \overline{\alpha}_t})^l, \quad l \geqslant 1,
\eeqn
and hence by \eqref{prod-exp}, for $m, n \geqslant 0,$ we get
\beqn
&& \sum_{s=0}^l (-1)^s {l \choose s}  B_{m+1+s} \,B^*_{n+1+s} \\
&=& \sum_{r, t=1}^k  \frac{\Big(\frac{1}{a_r \overline{a_t}}   \sum_{j=1}^k p_{j}(\alpha_r) \overline{p_{j}(\alpha_t)} \Big)}{\alpha^{m+2}_r \overline{\alpha}^{n+2}_t}(1-\frac{1}{\alpha_r \overline{\alpha}_t})^l.
\eeqn
This together with (i)$\Leftrightarrow$(ii) of Theorem \ref{main-thm-0} completes the proof.
\end{proof}

In the remaining part of this section, we present some applications of Theorem \ref{higher-r-thm}. 
\begin{proof}[Proof of Corollary \ref{coro-p-zero}]
By \eqref{orthogonality-p}, 
\beqn
\sum_{r, t=1}^k  \Big(\frac{1}{a_r \overline{a_t}}   \sum_{j=1}^k p_{j}(\alpha_r) \overline{p_{j}(\alpha_t)} \Big) (1-\frac{1}{\alpha_r \overline{\alpha}_t})^l \Big(\!\Big(\frac{1}{\alpha^{m+2}_r \overline{\alpha}^{n+2}_t}\Big)\!\Big)_{m, n \geqslant 0} \\ = \sum_{r=1}^k  \frac{(1-\frac{1}{|\alpha_r|^2})^l}{|a_r|^2|\alpha_r|^4 }   \sum_{j=1}^k |p_{j}(\alpha_r)|^2  \Big(\!\Big(\frac{1}{\alpha^{m}_r \overline{\alpha}^{n}_r}\Big)\!\Big)_{m, n \geqslant 0}. 
\eeqn
Since $|\alpha_r| > 1,$ in view of Theorem \ref{higher-r-thm}, it suffices to check that $\Big(\!\!\Big(\frac{1}{\alpha^{m}_r \overline{\alpha}^{n}_r}\Big)\!\!\Big)_{m, n \geqslant 0}$ is formally positive semi-definite for every $r=1, \ldots, k$.
Since this matrix is equal to $V_rV^*_r$ with $V_r$ denoting the column vector $\Big(\!\!\Big(\frac{1}{\alpha^m_r}\Big)\!\!\Big)_{m \geqslant 0},$ the desired conclusion is immediate.
\end{proof}

We provide below a necessary condition for the subnormality of the Cauchy dual of the multiplication operator on a de Branges-Rovnyak space.
\begin{corollary} \label{coro-posi}
Assume that  the hypotheses of Theorem \ref{higher-r-thm} hold. If $\mathscr M'_z$ is subnormal,  then $$\sum_{r, t=1}^k \Big(\frac{1}{\alpha^2_r \overline{\alpha}^2_t a_r \overline{a_t}}   \sum_{j=1}^k p_{j}(\alpha_r) \overline{p_{j}(\alpha_t)} \Big) \delta_{\!\frac{1}{\alpha_r \overline{\alpha}_t}}$$ is a positive measure supported in $[0, 1].$ 
\end{corollary}
\begin{proof}
Assume that $\mathscr M'_z$ is subnormal and let
\beqn
\gamma_m = \sum_{r, t=1}^k \Big(\frac{1}{a_r \overline{a_t}}   \sum_{j=1}^k p_{j}(\alpha_r) \overline{p_{j}(\alpha_t)} \Big) \frac{1}{\alpha^{m+2}_r \overline{\alpha}^{m+2}_t}, \quad m \geqslant 0.
\eeqn
By Theorem \ref{higher-r-thm} and \eqref{prod-exp}, the sequence $\{\gamma_m\}_{m \geqslant 0}$ is completely monotone. Hence, by the solution of the Hausdorff moment problem (see \cite{BCR}), there exists a finite positive Borel measure $\nu$ on $[0, 1]$ such that 
\beqn
\gamma_m = \int_{0}^1 t^m d\nu(t), \quad m \geqslant 0.
\eeqn 
On the other hand, if $K=[0, 1] \cup \{\frac{1}{\alpha_r \overline{\alpha}_t} : 1 \leqslant r, t \leqslant k\}$ and $$\mu := \sum_{r, t=1}^k \Big(\frac{1}{\alpha^2_r \overline{\alpha}^2_t a_r \overline{a_t}}   \sum_{j=1}^k p_{j}(\alpha_r) \overline{p_{j}(\alpha_t)} \Big) \delta_{\!\frac{1}{\alpha_r \overline{\alpha}_t}},$$ then we get
\beqn
\gamma_m = \int_{K} z^m d\mu, \quad m \geqslant 0.
\eeqn
It follows that 
\beqn
\int_{0}^1 t^m d\nu(t) = \int_{K} z^m d\mu, \quad m \geqslant 0.
\eeqn
Since $K$ is a compact set with connected complement in $\mathbb C,$ by Mergelyan's Theorem (see \cite[20.5 Theorem]{Ru}), any continuous function on $K$ can be approximated uniformly by polynomials in $z.$ By the Riesz representation theorem (see \cite[6.19 Theorem]{Ru}), $\mu$ is necessarily supported in $[0, 1]$ and it coincides with $\nu.$
\end{proof}
\begin{remark} \label{rmk-last} 
For $1 \leqslant r, t \leqslant k,$ let $A_{r, t} = \{(u, v) : \alpha_u \overline{\alpha}_v = \alpha_r \overline{\alpha}_t \}$ and
$$c_{r, t}:= \sum_{(u, v) \in A_{r, t}}  \Big(\frac{1}{\alpha^2_u \overline{\alpha}^2_v a_u \overline{a_v}}   \sum_{j=1}^k p_{j}(\alpha_u) \overline{p_{j}(\alpha_v)} \Big).$$
If the hypotheses of Theorem \ref{higher-r-thm} hold, then the following possibilities occur:
\begin{enumerate}
\item[(i)] If $\frac{1}{\alpha_r \overline{\alpha}_t} \notin [0, 1],$ then $c_{r, t} =0.$
\item[(ii)] If $\frac{1}{\alpha_r \overline{\alpha}_t} \in [0, 1],$ then $c_{r, t} \geqslant 0.$
\end{enumerate}
\end{remark}

Combining Corollary \ref{coro-p-zero} with Remark \ref{rmk-last}(i) yields the following.
\begin{corollary} \label{coro-p-zero-iff}
Assume the hypotheses of Theorem \ref{higher-r-thm}. If ${\alpha_r \overline{\alpha}_t}$'s are distinct complex numbers belonging to $\mathbb C \setminus [1, \infty)$ for every $1 \leqslant r \neq t \leqslant k,$ then $\mathscr M'_z$ is subnormal if and only if 
\beqn
\sum_{j=1}^k p_j(\alpha_r) \overline{p_j(\alpha_t)} = 0, \quad 1 \leqslant r \neq t \leqslant k.
\eeqn
\end{corollary}

\section{Classical de Branges-Rovnyak spaces}

Let $b$ be a non-extreme point of the closed unit ball of $H^{\infty}(\mathbb D).$ 
It is well-known that there exists a unique outer function $a \in H^{\infty}(\mathbb D)$ (that is, $\bigvee \{z^n a : n \geqslant 0\}=H^2(\mathbb D)$)  such that $|a|^2+|b|^2=1$ almost everywhere on unit circle and $a(0) >0$ (see \cite[Chapter 23, Section 1]{FM}). We refer to $a$ as the {\it mate} of $b.$ The following lemma provides a formula for the reproducing kernels of the so-called the Cauchy dual of classical de Branges-Rovnyak spaces.
\begin{lemma} \label{Cauchy dual space}
Let $b$ be a nonextreme point of the closed unit ball of $H^{\infty}$ such that $b(0)=0$ and let $a$ be the mate of $b.$  Let $\phi = \frac{b}{a}.$ 
Then the Cauchy dual $\mathscr M'_z$ of $\mathscr M_z$ is unitarily equivalent to the operator of multiplication by $z$ on the reproducing kernel Hilbert space $\mathscr H_{\kappa'_b},$ where $\kappa'_b$ is given by 
\beqn
 \kappa'_b(z, w) 
 = \frac{1 + \phi(z)\overline{\phi(w)}}{1-z\overline{w}}, \quad z, w \in \mathbb D.
 \eeqn
\end{lemma}
\begin{proof}  Let $\phi(z) = \sum_{j=0}^{\infty}c_j z^j$ for $z \in \mathbb D.$ By \cite[Lemma 3.2]{CGR}, 
\beq \label{gram}
\inp{z^m}{z^n}_{\mathcal H(b)} = \begin{cases} \delta_{m, n}   + \sum_{k=0}^n \overline{c_{ m-n+k}} {c_{ k}} & \mbox{if~} m \geqslant n, \\
 \sum_{j=0}^m \overline{c_{ j}} {c_{ n-m+j}} & \mbox{if~} m < n,
 \end{cases}
 \eeq
 Since $\ker \mathscr M^*_z$ is spanned by $1$ and $\mathscr M_z$ is cyclic with cyclic vector $1$ (see \cite{FM} and \cite[Theorems 1.3 and 5.5]{AM}), by \cite[Corollary 2.8]{S}, $\mathscr M'_z$ is analytic. 
 One may now apply \cite[Corollary 2.14]{S} to $T=\mathscr M'_z$ to conclude that $\kappa'_b$ is given by
 \beqn
 \kappa'_b(z, w) = \sum_{m, n \geqslant 0} \inp{z^m}{z^n}_{\mathcal H(b)} z^n \overline{w}^m, \quad z, w \in \mathbb D.
 \eeqn
It is now easy to see using \eqref{gram} that $\kappa'_b$ has the desired formula (see the proof of Theorem \ref{main-thm-0} for a similar argument). 
\end{proof}

The main result of this section provides an affirmative solution of the Cauchy dual subnormality problem for classical de Branges-Rovnyak spaces (cf. \cite[Corollary 3.6]{BS} and \cite[Theorem 3.3]{ACJS}).
\begin{theorem} \label{main-thm}
Let $b$ be a nonextreme point of the closed unit ball of $H^{\infty}(\mathbb D)$ such that $b(0)=0.$ If $\mathscr M_z$ on $\mathcal H(b)$ is concave, then the Cauchy dual $\mathscr M'_z$ of $\mathscr M_z$ is a subnormal contraction. Moreover, the following hold: 
\begin{enumerate}
\item[(i)] The sequence $\big\{\frac{\partial^m\overline{\partial}^n \kappa_b(0, 0)}{m!n!}\big\}_{m, n \geqslant 0}$ is a complex moment sequence with the representing measure $\mu$ given by 
\beqn
\Big(1 - \nu \Big(2\,\Re 
\Big(\frac{1}{1-e^{-i \theta}\beta}\Big)-1\Big)\Big) \frac{d\theta}{2\pi} +  \nu \, \delta_\beta,
\eeqn
where $\nu =   \frac{|\gamma|^2}{1-|\beta|^2}$ for some scalars $\gamma, \beta \in \mathbb C$ with $|\beta| < 1.$ 
\item[(ii)] $\mathscr M'_z$ is unitarily equivalent to the operator of multiplication by $z$ on the reproducing kernel Hilbert space $\mathscr H_{\kappa'_b},$ where $\kappa'_b$ is given by 
\beqn
 \kappa'_b(z, w) 
 = \frac{1 +  \frac{ |\gamma|^2 z \overline{w}}{(\rho - \sigma z)(\overline{\rho} - \overline{\sigma w})} }{1-z\overline{w}}, \quad z, w \in \mathbb D
 \eeqn
 for some scalars $\sigma, \rho \in \mathbb C.$
\end{enumerate}
\end{theorem}
\begin{proof}
Assume that $\mathscr M_z$ on $\mathcal H(b)$ is concave.  By \cite[Theorem 1]{KZ} and the assumption $b(0)=0,$ there exist scalars $\gamma, \beta \in \mathbb C$ such that $|\beta| < 1$ and 
\beqn
{b}(z) = \frac{\gamma z}{1-\beta z}, \quad z \in \mathbb D.
\eeqn
It now follows from Corollary \ref{coro-p-zero} that $\mathscr M'_z$ is a subnormal contraction.

Note that 
$b(z)=  \sum_{n \geqslant 1}b_{n}z^n,$ where $b_n$ is given by
\beq \label{b-lambda}
b_{n} = \gamma
\beta^{n-1}, \quad n \geqslant 1.
\eeq
To see (i), note that
by \eqref{formula-coeff} and \eqref{b-lambda},
\beqn
\frac{\partial^m\overline{\partial}^n \kappa_b(0, 0)}{m!n!} &=& \begin{cases} \delta_{m, n} - |\gamma|^2 \beta^{ m-n} \frac{1-|\beta|^{2n}}{1-|\beta|^2}  & \mbox{if~} m \geqslant n, \\
\delta_{m, n}  -  |\gamma|^2 \overline{\beta}^{^{n-m}} \frac{1-|\beta|^{2m}}{1-|\beta|^2}  & \mbox{if~} m < n.
 \end{cases} \notag 
\eeqn
Since for any integer $\ell \geqslant 0,$
\beq \label{exp-real}
2\int_{\mathbb T} e^{i (m-n) \theta } \Re \Big(\frac{e^{-i \ell \theta}}{1-e^{-i \theta}\beta}\Big)\frac{d \theta}{2\pi} = \begin{cases} 
\beta^{m-n-\ell} & \mbox{if~} m > n, \\
2\delta_{0, \ell} & \mbox{if~}m=n,\\ 
\overline{\beta}^{n-m-\ell} & \mbox{if~} m < n.
\end{cases}
\eeq
it can be easily seen that $\mu,$ as given in part (i), is the representing measure of $\big\{\frac{\partial^m\overline{\partial}^n \kappa_b(0, 0)}{m!n!}\big\}_{m, n \geqslant 0}.$  

To see part (ii), note that by \cite[Lemma 6]{KZ}, the mate $a$ of $b$ is given by 
\beqn
a(z) = \frac{\rho - \sigma z}{1-\beta z}, \quad z \in \mathbb D,
\eeqn
for some $\rho, \sigma \in \mathbb C.$ The desired formula now 
follows from Lemma \ref{Cauchy dual space}.
\end{proof}
\begin{remark}
The conclusion of Theorem \ref{main-thm} can be extended to Schur functions which do not necessarily vanish at the origin. Indeed, by \cite[Lemma 4.7]{LGR}, $\mathscr M_z$ on $\mathcal H(b)$ is unitarily equivalent to the operator ${\mathscr M}_z$ of multiplication by $z$ on $\mathcal H(\tilde{b}),$ where $\tilde{b}(0)=0.$ Thus the Cauchy dual $\mathscr M'_z$ of any concave multiplication operator $\mathscr M_z$ on $\mathcal H(b)$ is a subnormal contraction. Moreover, by \cite[Theorem 1]{KZ},
\beqn
{b}(z) = \frac{c+\gamma z}{1-\beta z}, \quad z \in \mathbb D
\eeqn
for some $c, \gamma, \beta \in \mathbb C$ with $|\beta| < 1.$ One may now argue as above using \eqref{exp-real} to see that the sequence $\big\{\frac{\partial^m\overline{\partial}^n \kappa_b(0, 0)}{m!n!}\big\}_{m, n \geqslant 0}$ is a complex moment sequence with the representing measure $\mu$ given by
\beqn
\Big(1 - (\nu +|c|^2)\Big(2\,\Re 
\Big(\frac{1}{1-e^{-i \theta}\beta}\Big)-1\Big) - 2 \gamma \,\overline{c}\, \Re \Big(\frac{e^{-i\theta}}{1-e^{-i \theta}\beta}\Big)\Big) \frac{d\theta}{2\pi} +  \nu \, \delta_\beta,
\eeqn
where $\nu =   \frac{|c\beta+\gamma|^2}{1-|\beta|^2}$ for some scalars $c, \gamma, \beta \in \mathbb C$ with $|\beta| < 1.$ 
\end{remark}

We conclude this section with an application to Dirichlet-type spaces associated with measures supported at a point.
\begin{corollary} \label{point-one}
Let $\lambda$ be a point in the unit circle $\mathbb T$ and let $\tau$ be a positive number. Then the Cauchy dual $\mathscr M'_z$ of $\mathscr M_z$ on the 
Dirichlet-type space $\mathscr D(\tau \delta_\lambda)$ is a subnormal contraction. 
\end{corollary}
\begin{proof}
By \cite[Theorem 3.1]{CGR}, there exist $\alpha \in \mathbb C$ with $|\alpha|^2 = \tau$ and $\eta \in (0, 1)$ satisfying $\eta  + 1/\eta = 2 + \tau$ such that
$\mathscr D(\tau \delta_\lambda)$ coincides with the de Branges-Rovnyak space $\mathcal H(b_{\alpha, \eta}),$ where $b_{\alpha, \eta}$ is given by
\beqn
b_{\alpha, \eta}(z) = \frac{\sqrt{\eta} \alpha \overline{\lambda} z }{1-\eta \overline{\lambda}z}, \quad z \in \mathbb D.
\eeqn
Since $\mathscr M_z$ acting on $\mathscr D(\tau \delta_\lambda)$ is a $2$-isometry (see \cite[Theorem 3.7]{R1}), 
by Theorem \ref{main-thm}, the Cauchy dual of $\mathscr M_z$ is a subnormal contraction. 
\end{proof}

\section{Dirichlet-type spaces associated with finitely supported measures}

In this section, we  discuss cyclic analytic $2$-isometries $T$ with finite rank defect operators $\Delta_T:=T^*T-I$ and their relationship with Dirichlet-type spaces associated with finitely supported measures.  
In particular, we classify cyclic analytic $2$-isometries with finite rank defect operators $\Delta_T$ and also $\Delta_T$-regular cyclic analytic $2$-isometries. 
Recall that for any $2$-isometry, $\Delta_T$ is a positive operator (see \cite[Lemma 1]{R}). 
Following \cite{MMS}, we say that a norm increasing operator $T$ is {\it $\Delta_T$-regular} if $\Delta_T T = \Delta^{1/2}_T T \Delta^{1/2}_T.$

\begin{theorem} \label{lem-finite-rank-rep}
Let $T \in \mathcal B(\mathcal H)$ be a cyclic analytic $2$-isometry and let $\Delta_T:=T^*T-I$. Then the following statements are true:
\begin{enumerate}
\item[(i)] Let $k$ be a positive integer. The rank of $\Delta_T$ is $k$ if and only if there exist distinct points $\zeta_1, \ldots, \zeta_k$ on the unit circle and positive numbers $c_1, \ldots, c_k$ such that $T$ is unitarily equivalent to the multiplication operator $\mathscr M_z$ on $\mathscr D(\mu),$ where $\mu= \sum_{j=1}^k c_j \delta_{\zeta_j}.$  
\item[(ii)] The operator $T$ is $\Delta_T$-regular
if and only if $T$ is unitarily equivalent to the multiplication operator $\mathscr M_z$ on $\mathscr D(c\, \delta_\zeta)$ for some scalar $c \geqslant 0$ and $\zeta \in \mathbb T.$
\end{enumerate}
\end{theorem}
\begin{proof} In view of the model theorem of Richter for cyclic analytic $2$-isometries (see \cite[Theorem 5.1]{R1}), we may assume that $T=\mathscr M_z$ acting on a Dirichlet-type space $\mathscr D(\mu)$ for some finite positive Borel measure $\mu$ on the unit circle $\mathbb T$. One may argue as in the proof of \cite[Theorem 1.26]{Ag-St} to see that with respect to the decomposition $\mathscr D(\mu) = \ker \Delta_{\mathscr M_z} \oplus \overline{\mbox{ran}\,\Delta_{\mathscr M_z}},$ $\mathscr M_z$ and $\Delta_{\mathscr M_z}$ decompose as follows:
\beq
 \label{matrix-d}
 \mathscr M_z = \Big(\begin{smallmatrix}
S & E \\
0 & W 
\end{smallmatrix}\Big),  \quad
\Delta_{\mathscr M_z} = \Big(\begin{smallmatrix}
0 & 0 \\
0 & D 
\end{smallmatrix} \Big),
\eeq
where $S$ is an isometry, $S^*E=0$ and $D$ is an injective positive operator such that $W^*DW=D.$

(i) Assume that $\mu= \sum_{j=1}^k c_j \delta_{\zeta_j}.$ 
It is well known that the point evaluation at $\zeta_j,$ $j=1, \ldots, k,$ is bounded on $\mathscr D(\mu)$ (see \cite[Corollary 2.3]{RS}). Thus there exist linearly independent vectors
 $\kappa_{\zeta_1}, \ldots,  \kappa_{\zeta_k} \in \mathscr D(\mu)$ such that 
 \beqn
 \inp{f}{\kappa_{\zeta_j}}=f(\zeta_j), \quad f \in \mathscr D(\mu), ~j=1, \ldots, k.  
 \eeqn
By \cite[Corollary 2.3]{RS}, for any $f \in \mathscr D(\mu),$
 \beqn
 \inp{\Delta_{\mathscr M_z} f}{f} = \|zf\|^2 - \|f\|^2 = \int_{\mathbb T} |f(\zeta)|^2 d\mu(\zeta) = \sum_{j=1}^k c_j |\inp{f}{\kappa_{\zeta_j}}|^2.
 \eeqn
It follows that 
 \beq \label{range-defect}
\mbox{ran}\,\Delta_{\mathscr M_z}  = \bigvee\{\kappa_{\zeta_j} : j=1, \ldots, k\}.
\eeq 
To see the converse, assume that the range of $\Delta_{\mathscr M_z}$ is $k$-dimensional.  Note that $D$ is an invertible operator on the finite dimensional Hilbert space $\mbox{ran}\,\Delta_{\mathscr M_z}.$ Hence, $D^{1/2}WD^{-1/2}$ is a unitary operator on $\mbox{ran}\,\Delta_{\mathscr M_z}.$ Consider the orthonormal basis of $\mbox{ran}\,\Delta_{\mathscr M_z}$ consisting of eigenvectors $v_j$ of $D^{1/2}WD^{-1/2}$ corresponding to the eigenvalue $\zeta_j \in \mathbb T,$ $j=1, \ldots, k.$   
By \cite[Corollary 2.3]{RS} and the polarization identity,
\beq \label{6.2}
\inp{\Delta_{\mathscr M_z} \mathscr M^n_z1}{1} = \inp{z^{n+1}}{z} - 
\inp{z^n}{1} = \int_{\mathbb T}\zeta^n d\mu(\zeta), \quad n \geqslant 0.
\eeq
Let $P$ denote the orthogonal projection of $\mathscr D(\mu)$ onto
$\mbox{ran}\,\Delta_{\mathscr M_z}.$  Since the set $\{D^{-1/2}v_j : j=1, \ldots, k\}$ forms a basis for $\mbox{ran}\,\Delta_{\mathscr M_z},$ there exist scalars $\lambda_j \in \mathbb C$ such that 
$P1 = \sum_{j=1}^k \lambda_j D^{-1/2}v_j.$ 
It now follows from \eqref{6.2} and \eqref{matrix-d} that for any integer $n \geqslant 0,$
\beqn
\int_{\mathbb T}\zeta^n d\mu(\zeta) &=& \inp{DW^nP1}{P1} \\
&=& \sum_{i, j=1}^k \lambda_i \overline{\lambda}_j \inp{D^{1/2}W^n D^{-1/2}v_i}{v_j} \\
&=& \sum_{i, j=1}^k \lambda_i \overline{\lambda}_j  \zeta^n_i \inp{v_i}{v_j} \\
&=& \sum_{i=1}^k |\lambda_i|^2  \zeta^n_i.
\eeqn
By the uniqueness of the trigonometric moment problem (which in turn follows from the density of trigonometric polynomials and the Riesz representation theorem \cite{Ru}), $\mu$ is equal to $\sum_{i=1}^k |\lambda_i|^2 \delta_{\zeta_i}.$ Since the range of $\Delta_{\mathscr M_z}$ is $k$-dimensional, by the first half, $\zeta_1, \ldots, \zeta_k$ are all distinct. This completes the proof of (i).

(ii) If $\mathscr M_z$ is an isometry, then the equivalence holds with $c=0.$
Suppose that $\mu = c\, \delta_\zeta$ for some scalar $c > 0$ and $\zeta \in \mathbb T.$ By (i), 
 $\Delta_{\mathscr M_z}$ is one-dimensional, and hence by \cite[Corollary 3.6]{BS},
$\mathscr M_z$ is $\Delta_{\mathscr M_z}$-regular. This yields the sufficiency part. 
To see the converse, assume now that $\Delta_{\mathscr M_z}$ is nonzero and $\mathscr M_z$ is $\Delta_{\mathscr M_z}$-regular. Hence, by
\cite[Proposition 5.1]{MMS}, $E$ as appearing in \eqref{matrix-d} is a one-to-one map from $\overline{\mbox{ran}\,\Delta_{\mathscr M_z}}$ into $\ker \Delta_{\mathscr M_z}$
such that $S^*E=0.$ However, by \cite[Theorem 3.2]{RS}, $\ker S^*$ is of dimension $1.$ It follows that the range of $E$ is at most of dimension $1.$ Since $E$ is injective, $\mbox{ran}\,\Delta_{\mathscr M_z}$ is one dimensional. The conclusion now follows from (i).
\end{proof}
\begin{remark} We note the following:
\begin{enumerate}
\item[(1)] It is worth mentioning that part (ii) above is applicable to the {\it Brownian shift} $B_{\sigma, e^{i \theta}}$ of covariance $\sigma > 0$ and angle $\theta,$ as introduced and studied in \cite{Ag-St} (the reader is referred to \cite[Section 5]{Ag-St} for the definition of the Brownian shift).
\item[(2)] Let $\mu$ be a finite Borel positive measure. If $\mathscr D(\mu)$ coincides with a de Branges-Rovnyak space $\mathcal H(B)$ of rank $k,$ with equality of norms, then it may be deduced from \cite[Lemma 5.1]{LGR} that $\Delta_{\mathscr M_z}$ is $k$-dimensional, and hence by Theorem \ref{lem-finite-rank-rep}, there exist $c_1, \ldots, c_k > 0$ and $\zeta_1, \ldots, \zeta_k \in \mathbb T$ such that $\mu = \sum_{j=1}^k c_j \delta_{\zeta_j}$ (cf. \cite[Proposition 2]{Sa} and \cite[Theorem 3.1]{CGR}).
\end{enumerate} 
\end{remark}

In operator theoretic terms, the following fact recovers
a special case of \cite[Corollary 3.6]{BS}. 
\begin{corollary}
Let $T$ be a cyclic analytic $2$-isometry on $\mathcal H.$ If the range of $T^*T-I$ is at most one-dimensional, then the Cauchy dual operator $T'$ of $T$ is a subnormal contraction.
\end{corollary}
\begin{proof}
If the range of $T^*T-I$ is at most one-dimensional, then by Theorem \ref{lem-finite-rank-rep}, $T$ is unitarily equivalent to the multiplication operator $\mathscr M_z$ on $\mathscr D(c\, \delta_\zeta)$ for some scalar $c \geqslant 0$ and $\zeta \in \mathbb T.$ The desired conclusion now follows from Corollary \ref{point-one}. 
\end{proof}

The second result of this section describes the de Branges-Rovnyak model space of the Dirichlet-type spaces associated with finitely supported measures. We capitalize below on the algorithm for computing the reproducing kernel for Dirichlet-type spaces associated with finitely supported measures, as presented in \cite{C}. A variant of this fact may also be deduced from \cite[Theorems 4.6 $\&$ 6.2 and Lemma 8.3]{LGR}. 
\begin{theorem} \label{Diri-de B} 
For positive scalars $c_1, \ldots, c_k$ and distinct points $\zeta_1, \ldots, \zeta_k$ on the unit circle $\mathbb T,$ consider the positive Borel measure $\mu = \sum_{j=1}^k c_j \delta_{\zeta_j},$ where $\delta_{\zeta_j}$ denotes the Dirac delta measure supported at $\zeta_j.$ 
Let $X(z)=(z, \ldots, z^k)^T$ and $\{e_j\}_{j=1}^k$ denote the standard basis of $\mathbb C^k.$ Then there exist $\alpha_1, \ldots, \alpha_k \in \mathbb C \setminus \overline{\mathbb D},$ unique up to permutation, and a $k \times k$ upper triangular matrix $P$ such that 
the Dirichlet-type space
$D(\mu)$ coincides with the de Branges-Rovnyak space $\mathcal H(B)$ with equality of norms, where $B=(\frac{p_1}{q}, \cdots, \frac{p_k}{q})$  and $$p_j(z) = \inp{PX(z)}{e_j} \quad j=1, \ldots, k, \quad q(z) = \prod_{j=1}^k (z-\alpha_j), \quad z \in \mathbb D.$$ 
Moreover, $\alpha_1, \ldots, \alpha_k$ are governed by 
\beq \label{Costara}
\prod_{j=1}^k |z-\zeta_j|^2 + \sum_{j=1}^k c_j \prod_{\substack{l=1, \ldots, k \\ l \neq j}}|z-\zeta_l|^2  = \gamma \prod_{j=1}^k |z-\alpha_j|^2, \quad z \in \mathbb T
\eeq 
 for some $\gamma >0.$  
\end{theorem}
\begin{proof} It has been observed in \cite{C} that there exist $\alpha_1, \ldots, \alpha_k \in \mathbb C \setminus \overline{\mathbb D}$ and $\gamma >0$ such that \eqref{Costara} holds. Two applications of maximum modulus theorem shows that  $\alpha_1, \ldots, \alpha_k \in \mathbb C \setminus \overline{\mathbb D}$ are unique up to permutation.
Let $p(z) = \frac{e^{i \theta}}{\sqrt{\gamma}}   \prod_{j=1}^k (z-\zeta_j),$ where $\theta \in \mathbb R$ is chosen such that $\frac{p(0)}{q(0)}> 0.$
Since the multiplication operator $\mathscr M_z$ on $\mathscr D(\mu)$ is an analytic norm increasing operator (see \cite[Theorem 3.6]{R1}) and the reproducing kernel for $\mathscr D(\mu)$ is normalized, by Lemma \ref{expansion-pd}, there exists a positive semi-definite kernel $\eta : \mathbb D \times \mathbb D \rar \mathbb C$ such that 
\beq \label{Richter}
\eta(z, 0) =0, \quad \kappa(z, w) = \frac{1- \eta(z, w)}{1-z\overline{w}}, \quad z, w \in \mathbb D.
\eeq
By \cite[Theorems 5.1 and 4.4]{C}, the reproducing kernel $\kappa(z, w)$ for $\mathscr D(\mu)$ is given by
\beq
\label{C-kernel}
\kappa(z, w) = \frac{O(z)\overline{O(w)}}{1-z\overline{w}} + \sum_{j=1}^k f_j(z)\beta_j(w),  \quad z, w \in \mathbb D,
\eeq
where $O, f_j, \beta_j$ are given by
\beqn
&& O(z) = \frac{p(z)}{q(z)}, \quad f_j(z) = \frac{O(z)}{O'(\zeta_j)(z-\zeta_j)}, \\ && 
\begin{pmatrix}
 \overline{\beta_1(w)}  \\
\vdots \\ 
\overline{\beta_k(w)}
\end{pmatrix}  = \Big( (\!(\inp{f_i}{f_j})\!)_{1 \leqslant i, j \leqslant k}\Big)^{-1} \begin{pmatrix}
 f_1(w)  \\
\vdots \\ 
f_k(w)
\end{pmatrix}.
\eeqn
Letting $(\!(b_{ij})\!)_{1 \leqslant i, j \leqslant k}$ denote the inverse of $\!(\inp{f_i}{f_j})\!)_{1 \leqslant i, j \leqslant k},$ we obtain
\beqn
\beta_j(w) = \sum_{i=1}^k \overline{b}_{ji} \, \overline{f_i(w)}, \quad j=1, \ldots, k.
\eeqn
This combined with \eqref{Richter} and \eqref{C-kernel} 
\beqn 
 \frac{1- \eta(z, w)}{1-z\overline{w}} 
&=& \frac{1}{1-z\overline{w}} \frac{p(z)\overline{p(w)}}{q(z)\overline{q(w)}} + \sum_{i, j=1}^k \overline{b}_{ji}  f_j(z)   \overline{f_i(w)} \\
&=& \frac{p(z)\overline{p(w)}}{q(z)\overline{q(w)}} \Big(\frac{1}{1-z\overline{w}}  + \sum_{i, j=1}^k \frac{\overline{b}_{ji}}{O'(\zeta_j)\overline{O'(\zeta_i)}}  \frac{1}{(z-\zeta_j)(\overline{w}-\overline{\zeta_i})} \Big).
\eeqn
This yields
\beqn 
&& q(z) \,\eta(z, w) \,\overline{q(w)} \\ &=& \! q(z)\overline{q(w)} - p(z)\overline{p(w)} \Big(1+ (1-z \overline{w}) \sum_{i, j=1}^k \frac{\overline{b}_{ji}}{O'(\zeta_j)\overline{O'(\zeta_i)}}  \frac{1}{(z-\zeta_j)(\overline{w}-\overline{\zeta_i})} \Big). \notag
\eeqn
Since the expression on the right hand side is a polynomial in $z$ and $\overline{w},$ 
there exists a matrix $\tilde{A}=(\!(a_{ij})\!)_{0 \leqslant i, j \leqslant k}$ such that
\beqn
q(z) \,\eta(z, w) \,\overline{q(w)} = \sum_{i, j=0}^k a_{ij} z^i \overline{w}^j, \quad z, w \in \mathbb D.
\eeqn
As $\eta(z, 0)=0$ for $z \in \mathbb D,$ 
\beqn
q(z) \,\eta(z, w) \,\overline{q(w)} = \sum_{i, j=1}^k a_{ij} z^i \overline{w}^j=\inp{AX(z)}{X(w)}, \quad z, w \in \mathbb D,
\eeqn
where $A$ is the $k \times k$ matrix obtained from $\tilde{A}$ by removing first row and first column.
Further, since $q(z) \,\eta(z, w) \, \overline{q(w)}$ is a positive semi-definite kernel,  the matrix $A$ is positive semi-definite (see Proposition \ref{pd-k-m}). By Cholesky's decomposition (see \cite[Pg 2]{B}), there exists a $k \times k$ upper triangular matrix $P$ such that $A=P^*P.$  
It follows that 
\beqn 
q(z)\, \eta(z, w) \, \overline{q(w)} &=& \notag \inp{AX(z)}{X(w)} \\ \notag
&=&  \inp{PX(z)}{PX(w)} \\ \notag 
&=& \notag 
\sum_{j=1}^k \inp{PX(z)}{e_j} \overline{\inp{PX(w)}{e_j}} \\
&=&
\sum_{j=1}^k p_j(z)\overline{p_j(w)}.
\eeqn
This shows that 
\beqn
\eta(z, w) = \sum_{j=1}^k \frac{p_j(z)}{q(z)}\frac{\overline{p_j(w)}}{\overline{q(w)}}, \quad z, w \in \mathbb D.
\eeqn
This completes the proof.
\end{proof}
\begin{remark} \label{ortho-rmk}
As noticed in the above proof,  
\beq \label{ortho-eqn}
&& \sum_{j=1}^k p_j(z)\overline{p_j(w)} \\ &=& q(z)\overline{q(w)} - p(z)\overline{p(w)} \Big(1+ \sum_{i, j=1}^k \frac{\overline{b}_{ji}}{O'(\zeta_j)\overline{O'(\zeta_i)}}  \frac{1-z \overline{w}}{(z-\zeta_j)(\overline{w}-\overline{\zeta_i})} \Big), \notag
\eeq
where the matrix $(\!(b_{ij})\!)_{1 \leqslant i, j \leqslant k}$ is the inverse of $(\!(\inp{f_i}{f_j}_{\mathscr D(\mu)})\!)_{1 \leqslant i, j \leqslant k}$ and $f_j(z) = \frac{O(z)}{O'(\zeta_j)(z-\zeta_j)},$ $j=1, \ldots, k.$ Thus, for $1 \leqslant r \neq t \leqslant k,$ we have
$\sum_{j=1}^k p_j(\alpha_r) \overline{p_j(\alpha_t)} =0$ if and only if 
\beqn \label{ortho-fact}
\sum_{i, j=1}^k \frac{\overline{b}_{ji}}{O'(\zeta_j)\overline{O'(\zeta_i)}}  \frac{1}{(\alpha_r-\zeta_j)(\overline{\alpha_t}-\overline{\zeta_i})} =\frac{1}{\alpha_r \overline{\alpha_t}-1}.
\eeqn
Note further that 
by \cite[Lemmas 4.2 and 4.3]{C}, $\inp{f_i}{f_j}$ is given by 
\beqn
\inp{f_i}{f_j} = \begin{cases} c_i \zeta_i f'_i(\zeta_i) & \mbox{if~} i=j, \\
\frac{1}{O'(\zeta_i)\overline{O'(\zeta_j)}(1-\zeta_i\overline{\zeta_j})} & \mbox{otherwise}.
\end{cases}
\eeqn
\end{remark}


\section{Dirichlet-type spaces associated with measures supported at antipodal points}

In this section, we apply the results obtained in the previous sections to solve affirmatively the Cauchy dual subnormality problem for 
Dirichlet-type spaces associated with measures supported at antipodal points.
We begin with a fact which allows us to reduce the above problem to the case in which the antipodal points are $1$ and $-1.$

\begin{proposition} \label{mu-zeta}
Let $\mu$ be a finite positive Borel measure on the unit circle $\mathbb T.$  For $\zeta \in \mathbb T,$ let $\mu_{\zeta}$ be the finite positive Borel measure defined by $\mu_{\zeta}(\Delta) = \mu(\zeta \Delta)$ for every Borel subset $\Delta$ of $\mathbb T.$ Let $\mathscr M_{z, \zeta}$ denote the operator of multiplication by $z$ on the Dirichlet-type space $\mathscr D(\mu_\zeta).$ Then $\mathscr M_{z, \zeta}$ is unitarily equivalent to $\overline{\zeta} \mathscr M_{z, 1}.$ In particular,  we have the following:
\begin{enumerate}
\item[(i)] The Cauchy dual operator $\mathscr M'_{z, \zeta}$ of $\mathscr M_{z, \zeta}$ is subnormal if and only if $\mathscr M'_{z, 1}$ is subnormal.
\item[(ii)] If $\kappa$ is the reproducing kernel for $\mathscr D(\mu),$ then the reproducing kernel $\kappa_\zeta$ for $\mathscr D(\mu_\zeta)$ is given by
\beqn
\kappa_\zeta(z, w) = \kappa(\zeta z, \zeta w), \quad z, w \in \mathbb D.
\eeqn
\end{enumerate}
\end{proposition}
\begin{proof}
Note that for any $z \in \mathbb D,$ by the $\mathbb T$-invariance of the Poisson kernel $P(z, \lambda) = \frac{1-|z|^2}{|z-\lambda|^2}$, $z \in \mathbb D,$ $\lambda \in \mathbb T,$ we get 
\beqn
P_{\mu_\zeta}(z) = \int_\mathbb T P(z, \lambda)d\mu_\zeta(\lambda) 
= \int_\mathbb T P(\zeta z, \zeta \lambda)d\mu(\zeta \lambda) 
= P_\mu(\zeta z).
\eeqn
It now follows from the rotation invariance of the area measure that for any $f \in \mathscr D(\mu_\zeta),$
\beqn
\|f\|^2_{\mathscr D(\mu_\zeta)} &=& \|f\|^2_{H^2(\mathbb D)} + \int_\mathbb D |f'(z)|^2 P_\mu(\zeta z)dA(z) \\
&=& \|f\|^2_{H^2(\mathbb D)} + \int_\mathbb D |f'(\overline{\zeta}z)|^2 P_\mu(z)dA(z) \\
&=& \|f_{\overline{\zeta}}\|^2_{\mathscr D(\mu)},
\eeqn
where $f_{\overline{\zeta}}(z)= f(\overline{\zeta}z),$ $z \in \mathbb D.$ This yields the unitary $U_\zeta : f \mapsto f_{\overline{\zeta}}$ from $\mathscr D(\mu_\zeta)$ onto $\mathscr D(\mu).$ It is easy to verify that $U_\zeta \mathscr M_{z, \zeta} = \overline{\zeta} \mathscr M_{z, 1} U_\zeta.$ Part (i) now follows from the definition of the Cauchy dual operator and the fact that a scalar multiple of a subnormal operator is again subnormal. Since 
$U_{\zeta}$ is a unitary, $\kappa(\zeta z, \zeta w),$ $z, w \in \mathbb D,$ is easily seen to be a reproducing kernel for $\mathscr D(\mu_\zeta),$ and hence 
by the uniqueness of the reproducing kernel, part (ii) follows.
\end{proof}

To complete the proof of Theorem \ref{rank-2-zeta}, in view of Proposition \ref{mu-zeta}, one may focus on the antipodal points $1$ and $-1.$ 
\begin{example} \label{anti-exam}
Consider the Dirichlet-type space $\mathscr D(\mu)$ associated with the measure $\mu=c_1 \delta_{1} + c_2 \delta_{-1},$ where $c_1, c_2$ are positive numbers. 
We already noted that the multiplication operator $\mathscr M_z$ is a cyclic $2$-isometry. Moreover, the range of $\mathscr M^*_z \mathscr M_z-I$ is of dimension $2$ (see Lemma \ref{lem-finite-rank-rep}(i) and \eqref{range-defect}).  
By Theorem \ref{Diri-de B}, there exist $\alpha_1, \alpha_2 \in \mathbb C \setminus \overline{\mathbb D}$ (unique up to permutation) and $\gamma > 0$ such that 
\beq \label{Riesz-id}
|z-1|^2|z+1|^2 + c_1 |z+1|^2 + c_2|z-1|^2 = \gamma |z-\alpha_1|^2 |z-\alpha_2|^2, ~~ z \in \mathbb T.
\eeq
A routine verification shows that the above values of $\gamma, \alpha_1, \alpha_2$ satisfy \eqref{Riesz-id}:\footnote{To see this, one may try to solve the above identity for real numbers $\alpha_1,$ $\alpha_2$ and $\gamma$ assuming that $\alpha_1 > 1$ and $\alpha_2 < -1,$ and 
evaluating \eqref{Riesz-id} at $z=\pm 1, i,$}
\beq \label{gamma-alpha-1}
\gamma &=& \notag \frac{(\sqrt{4 + c^2_+} - c_+)^2}{4} = -\frac{1}{\alpha_1 \alpha_2},  \\
\alpha_1 &=& \notag \frac{(c_- + \sqrt{4 + c^2_-})(c_+ + \sqrt{4 + c^2_+})}{4}, \\ 
\alpha_2 &=& \frac{(c_- - \sqrt{4 + c^2_-})(c_+ + \sqrt{4 + c^2_+})}{4},
\eeq
where $c_+ =\sqrt{c_1}+\sqrt{c_2}$ and $c_{-} =\sqrt{c_1}-\sqrt{c_2}.$
By Proposition \ref{Diri-de B}, $D(\mu)$ coincides with the de Branges-Rovnyak space $\mathcal H(B)$ with $B=(b_1, b_2)$ and $b_j=\frac{p_j}{(z-\alpha_1)(z-\alpha_2)}$ for some degree $2$ polynomials $p_1$ and $p_2.$  
We contend that 
\beq \label{sum-pj-rank2}
\sum_{j=1}^2 p_j(\alpha_1) \overline{p_j(\alpha_2)} =0. 
\eeq
In view of Remark \ref{ortho-rmk}, it suffices to check 
\beq \label{ortho-fact-new}
\sum_{i, j=1}^2 \frac{\overline{b}_{ji}}{O'(\zeta_j)\overline{O'(\zeta_i)}}  \frac{1}{(\alpha_1-\zeta_j)({\alpha_2}-\overline{\zeta_i})} =\frac{1}{\alpha_1 {\alpha_2}-1},
\eeq
where $\zeta_1=1,$ $\zeta_2=-1.$  
After some routine calculations using Remark \ref{ortho-rmk}, we obtain
\beqn \label{clumsy}
&& O(z) = \frac{1}{\sqrt{\gamma}} \frac{z^2-1}{(z-\alpha_1)(z-\alpha_2)}, \quad
 O'(-1) = \frac{1}{\sqrt{c_2}}, \quad  O'(1) =-\frac{1}{\sqrt{c_1}}, \notag
\\
&& f_1(z) = -\frac{\sqrt{c_1}}{\sqrt{\gamma}} \frac{z+1}{(z-\alpha_1)(z-\alpha_2)}, \quad f_2(z) = \frac{\sqrt{c_2}}{\sqrt{\gamma}}  \frac{z-1}{(z-\alpha_1)(z-\alpha_2)}, \\
&& f'_1(1) = \frac{-3 +\alpha_1 + \alpha_2 + \alpha_1 \alpha_2}{2(1-\alpha_1)(1-\alpha_2)}, \quad
 f'_2(-1) = \frac{3 +\alpha_1 + \alpha_2 - \alpha_1 \alpha_2}{2(1+\alpha_1)(1+\alpha_2)},  \quad \quad 
\\ \notag
&& (\!(b_{ij})\!)_{1 \leqslant i, j \leqslant 2} = \frac{1}{-c_1c_2(f'_2(-1) f'_1(1) + \frac{1}{4})}\begin{pmatrix}
-c_2 f'_2(-1)  &  \frac{1}{2} \sqrt{c_1}\sqrt{c_2}\\
\frac{1}{2} \sqrt{c_1}\sqrt{c_2} &  c_1 f'_1(1)
\end{pmatrix}.
\eeqn

To verify \eqref{ortho-fact-new}, note that
\beqn
&& \sum_{i, j=1}^2 \frac{\overline{b}_{ji}}{O'(\zeta_j)\overline{O'(\zeta_i)}}  \frac{1}{(\alpha_1-\zeta_j)(\overline{\alpha_2}-\overline{\zeta_i})} \\
&=& \frac{1}{-(f'_2(-1) f'_1(1) + \frac{1}{4})} \Big( \frac{ -f'_2(-1) }{(\alpha_1-1)(\alpha_2-1)}  -  \frac{1}{2(\alpha_1-1)(\alpha_2+1)} \\
&-&   \frac{1}{2(\alpha_1+1)(\alpha_2-1)} +  \frac{f'_1(1)}{(\alpha_1+1)(\alpha_2+1)}\Big) \\
&=& \frac{2
 }{(f'_2(-1) f'_1(1) + \frac{1}{4})(\alpha^2_1-1)(\alpha^2_2-1)}. 
\eeqn
Since
$f'_2(-1) f'_1(1) + \frac{1}{4} =  \frac{-2(1 - \alpha_1 \alpha_2)}{(\alpha^2_1-1)(\alpha^2_2-1)},$
we obtain 
\beqn
\sum_{i, j=1}^2 \frac{\overline{b}_{ji}}{O'(\zeta_j)\overline{O'(\zeta_i)}}  \frac{1}{(\alpha_1-\zeta_j)(\overline{\alpha_2}-\overline{\zeta_i})}   = \frac{1}{\alpha_1\alpha_2-1}.
\eeqn
Thus the claim stands verified. This completes the verification of \eqref{sum-pj-rank2}. Hence, by Corollary \ref{coro-p-zero}, the Cauchy dual $\mathscr M'_z$ of $\mathscr M_z$ is subnormal. \eop 
\end{example}

We are now in a position to complete the proof of Theorem \ref{rank-2-zeta}.
\begin{proof}[Proof of Theorem \ref{rank-2-zeta}]
If $c_1 =0$ and $c_2=0,$ then $\mathscr M_z$ is an isometry and hence it is a subnormal contraction. If either $c_1$ or $c_2$ is $0,$ then the conclusion follows from Corollary \ref{point-one}. The case in which $c_1$ and $c_2$ are positive follows from Proposition \ref{mu-zeta} and Example \ref{anti-exam}.
\end{proof}


We conclude the paper with a modest generalization of \cite[Proposition 2]{Sa}  (cf. \cite[Theorem 3.1]{CGR} and \cite[Example 11.1]{LGR}).
\begin{proposition} 
\label{Sarason-2}
For positive scalars $c_1, c_2,$ consider the positive Borel measure $\mu = c_1 \delta_{1} + c_2 \delta_{-1}.$  
Then
the Dirichlet-type space
$D(\mu)$ coincides with the de Branges-Rovnyak space $\mathcal H(B)$ with $B=(b_1, b_2)$ and $b_j=\frac{p_j}{q},$ where $$p_1(z) = \gamma_1 z + \gamma_2 z^2, \quad p_2(z)=\gamma_3 z^2, \quad q(z) =  (z-\alpha_1)(z-\alpha_2), \quad z \in \mathbb D$$ 
for scalars $\alpha_i$ as given by \eqref{gamma-alpha-1}, and $\gamma_j$ are given by
\beq \label{gamma's}
\gamma_1  &=& \sqrt{(\alpha_1 + \alpha_2)^2 + 
\frac{\alpha_1 \alpha_2(1-\alpha^2_1)(1-\alpha^2_2)}{(-1+\alpha_1\alpha_2)}} \\ \notag
\gamma_2   &=& -\frac{1}{\overline{\gamma_1}}(\alpha_1 +\alpha_2)\Big(1 - \frac{\alpha_1 \alpha_2(-2 + \alpha_1 \alpha_2)}{2(-1+\alpha_1\alpha_2)} \Big) \\ \notag 
 \gamma_3  
&=& \sqrt{1 -  \frac{\alpha_1 \alpha_2(3 -\alpha_1 \alpha_2 - \alpha^2_1 - \alpha^2_2)}{-1+\alpha_1\alpha_2} - |\gamma_2|^2}
\eeq
\end{proposition}
\begin{proof}
By Proposition \ref{Diri-de B}, one may choose $p_1$ and $p_2$ of the form $$p_1(z)=\gamma_1 z + \gamma_2 z^2, \quad p_2(z)=\gamma_3 z^2,$$ which satisfy \eqref{ortho-eqn}. We now compute $\gamma_1, \gamma_2, \gamma_3.$ By \eqref{ortho-eqn}, we obtain
\beqn
&& |\gamma_1|^2 z\overline{w} + \gamma_2 \overline{\gamma_1} z^2 \overline{w}  + \gamma_1 \overline{\gamma_2} z \overline{w}^2 + (|\gamma_2|^2 + |\gamma_3|^2) z^2\overline{w}^2 
\\
&=&
\sum_{j=1}^2 p_j(z)\overline{p_j(w)} \\ &=& q(z)\overline{q(w)} - p(z)\overline{p(w)} \Big(1+ \sum_{i, j=1}^2 \frac{\overline{b}_{ji}}{O'(\zeta_j)\overline{O'(\zeta_i)}}  \frac{1-z \overline{w}}{(z-\zeta_j)(\overline{w}-\overline{\zeta_i})} \Big) \\
&=& (z-\alpha_1)(z-\alpha_2) (\overline{w}-\alpha_1)(\overline{w}-\alpha_2) - \frac{1}{\gamma} (z^2-1)(\overline{w}^2-1) \\
&-&
\frac{1}{\gamma}\frac{1-z\overline{w}}{f'_2(-1)f'_1(1)+1/4} \Big(f'_2(-1)(z+1)(\overline{w}+1) - f'_1(1)(z-1)(\overline{w}-1) \\ 
&+& \frac{1}{2} (z-1)(\overline{w}+1)  + \frac{1}{2} (z+1)(\overline{w}-1)\Big)  
\eeqn
Comparing the coefficients on both sides, we obtain
\beqn
|\gamma_1|^2  &=& (\alpha_1 + \alpha_2)^2 - \frac{2}{\gamma}\frac{1}{f'_2(-1)f'_1(1)+1/4} \\
\gamma_2 \overline{\gamma_1}  &=& -\alpha_1 -\alpha_2 + \frac{1}{\gamma}\frac{1}{f'_2(-1)f'_1(1)+1/4}\Big(f'_2(-1)+f'_1(1)\Big) \\
|\gamma_2|^2 + |\gamma_3|^2 &=&1 -\frac{1}{\gamma} + \frac{1}{\gamma}\frac{1}{f'_2(-1)f'_1(1)+1/4}\Big(f'_2(-1) - f'_1(1) + 1\Big)
\eeqn
It follows from computations of Example \ref{anti-exam} that
\beqn
f'_1(1) f'_2(-1) + 1/4 
&=& \frac{2(\alpha_1\alpha_2-1)}{(1-\alpha^2_1)(1-\alpha^2_2)}.
\eeqn
\beqn
f'_1(1) +  f'_2(-1)  
&=&  \frac{2(\alpha_1 \alpha_2-1)(\alpha_1+\alpha_2)}{(1-\alpha^2_1)(1-\alpha^2_2)} 
\eeqn
\beqn
 f'_2(-1) -f'_1(1) + 1 
&=& \frac{2(2 - \alpha^2_1 -  \alpha^2_2)}{(1-\alpha^2_1)(1-\alpha^2_2)}.
\eeqn
Since $\gamma = -\frac{1}{\alpha_1\alpha_2}$ (see \eqref{gamma-alpha-1}), it is easy to see that \eqref{gamma's} holds.
This gives the coefficients of $p_1$ and $p_2.$
\end{proof}
It is worth noticing that if $c_1= c_2,$ then $\alpha_1 + \alpha_2 =0$ (see  \eqref{gamma-alpha-1}), and hence $p_1, p_2$ and $q$ takes the form
$$p_1(z) = \gamma_1 z, \quad p_2(z)=\gamma_3 z^2, \quad q(z) =  z^2-\alpha^2_1, \quad z \in \mathbb D.$$
A particular case of this has been pointed out in \cite[Example 11.1]{LGR}.

\vskip.1cm

\noindent
{\bf Acknowledgments.} We convey our sincere thanks to Shuaibing Luo and V. M. Sholapurkar for providing several inputs that helped us in improving the earlier draft of this paper. 

{}

\end{document}